\documentclass[10pt,oneside,a4paper]{article}
\usepackage{titlesec}
\usepackage{authblk}
\titleformat{\subsection}[runin]{\normalfont\bfseries}{\thesubsection.}{.5em}{}[.~ ]
\titlespacing{\subsection}{0pt}{1.5ex plus .1ex minus .2ex}{0pt}
\usepackage{graphicx}
\usepackage{subcaption}
\usepackage{wrapfig}
\usepackage{amsfonts,amssymb,amsmath,amsthm}
\usepackage{enumerate}
\usepackage{url}
\usepackage[mathcal,mathscr]{euscript}
\usepackage[dvipsnames]{xcolor}
\usepackage{hyperref}
\hypersetup{colorlinks=true,linkcolor=Magenta,citecolor=PineGreen}
\usepackage{float}
\usepackage{tikz-cd}
\usepackage[most]{tcolorbox}
\usepackage{epigraph}

\theoremstyle{plain}
\newtheorem{thm}[equation]{Theorem}

\newtheorem{lemma}[equation]{Lemma}

\newtheorem{prop}[equation]{Proposition}
\newtheorem{cor}[equation]{Corollary}

\theoremstyle{definition}

\newtheorem{defi}[equation]{Definition}

\newtheorem{rmk}[equation]{Remark}
\newcommand{\pp}{\pmb{p}}

\newcommand{\dd}{\mathrm{d}}

\newcommand{\tr}{\mathrm{\mathrm{tr}}}

\newcommand{\SU}{\mathrm{\mathrm{SU}}}
\newcommand{\PU}{\mathrm{\mathrm{PU}}}
\newcommand{\SO}{\mathrm{\mathrm{SO}}}
\newcommand{\OO}{\mathrm{\mathrm{O}}}

\newcommand{\real}{\mathrm{\mathrm{Re}\,}}
\newcommand{\imag}{\mathrm{\mathrm{Im}\,}}
\newcommand{\PP}{\mathbb{P}}
\newcommand{\HH}{\mathbb{H}}
\newcommand{\BB}{\mathbb{B}}
\newcommand{\CC}{\mathbb{C}}
\newcommand{\ZZ}{\mathbb{Z}}
\newcommand{\KK}{\mathbb{K}}
\newcommand{\SP}{\mathbb{S}}
\newcommand{\RR}{\mathbb{R}}
\newcommand{\QQ}{\mathbb{Q}}

\newcommand{\orb}{\mathrm{orb}}

\newcommand{\TT}{\mathrm{T}}

\newcommand{\pf}{\mathrm{pf}}
\newcommand{\stab}{\,\mathrm{stab}}

\title{\vspace{-15mm}\fontsize{16pt}{10pt}\selectfont\textbf{Orbifolds and orbibundles in complex hyperbolic geometry}}

\author{
	\large
	\textsc{Hugo Cattarucci Botós}\thanks{Supported by S\~ao Paulo Research Foundation (FAPESP)}\\
	\normalsize \href{}{hugocbotos@usp.br}\\
	\normalsize{Departamento de Matem\'atica, ICMC, Universidade de S\~ao Paulo, S\~ao Carlos, Brasil}	
	\vspace{-5mm}
}
\date{}

\begin{document}
	
\maketitle

\begin{center}
	\large\textbf{Abstract}
\end{center}
	We develop the theory of orbibundles from a geometrical viewpoint using diffeology. One of our goals is to present new tools allowing to calculate invariants of complex hyperbolic disc orbibundles over $2$-orbifolds appearing in the geometry of $4$-manifolds. These invariants are the Euler number of disc orbibundles and the Toledo invariant of $\PU(2,1)$-representations of $2$-orbifold groups. 

	\section{Introduction}
	We study orbibundles via diffeology (following the suggestion in \cite[pag.\,94]{igl2}). Using this framework, we describe essential invariants appearing in complex hyperbolic geometry and prove orbifold generalizations of some classical results in the area. These new tools were developed while investigating the complex variant of the Gromov-Lawson-Thurston conjecture.
	
	A Riemannian manifold $N$ is uniformized by a simply connected complete Riemannian manifold $M$ if there is a Riemannian covering of $N$ by $M$. Uniformization plays an important role in classifying manifolds in dimensions $2$ and $3$ respectively via uniformization of Riemann surfaces and Thurston's geometrization conjecture (proved by Perelman in 2006). 
	
	Uniformization in dimension $4$ is far from being well understood. In this regard, it is natural to investigate disc bundles over surfaces, one of the simplest types of $4$-manifolds. Along these lines stands the {\bf Gromov-Lawson-Thurston conjecture:} An oriented disc bundle $M \to B$ over an oriented, connected, compact surface $B$ with negative Euler characteristic is uniformized by the real hyperbolic space $\HH_\RR^4$ (i.e., the usual hyperbolic $4$-space) if, and only if, $|e_R(M)|\leq 1$, where $e_R(M)$ is the relative Euler number of the bundle, defined as the quotient of the Euler number $e(M)$ of the disc bundle by the Euler characteristic $\chi(B)$ of the surface (see \cite{GLT}). Observe that the relative Euler number is preserved under pullbacks of $M$ by finite covers of $B$. Both directions of the conjecture are open (see \cite{ach}, \cite{GLT}, \cite{kap1}, and \cite{kui1} for details).
	
	The Euler number of an oriented disc bundle $M \to B$ is the oriented intersection number (see~\cite[Chapter 3]{pollack}) of two transversal sections and the Euler characteristic is the Euler number of the tangent bundle of the surface. Since an oriented disc bundle $M\to B$ is determined, up to isomorphism, by its Euler number, the relative Euler number measures how different from the tangent bundle $\mathrm{T}B \to B$ the disc bundle $M\to B$ is.
	Up to bundle isomorphism, there are three distinguished cases: the tangent, the cotangent, and the trivial disc bundles over $B$. The relative Euler numbers are respectively $1,-1,0$.
	  
	In 2011, new examples of complex hyperbolic disc bundles $M\to B$ were discovered \cite{discbundles}. Complex hyperbolic means that the total space $M$ is uniformized by the complex hyperbolic plane~$\HH_\CC^2$ (see Subsection \ref{subsection complex hyperbolic geometry}). They satisfy $|e_R|\leq 1$ and, therefore, support the {\bf complex GLT-conjecture} (same statement with $\HH_\CC^2$ in place of $\HH_\RR^4$). A large number of examples backing the conjecture can also be found in \cite{GKL}, \cite{sashanicolai}, and \cite{bgr}, including the above distinguished cases. The fact that the statement of the conjecture is the same in both cases may be seen as evidence that the conjecture is more about negative curvature than about constant negative curvature.
	
	In this paper, we develop a theory of $\SP^1$-orbibundles over compact oriented $2$-orbifolds via diffeology and use it to introduce the concept of Euler number for oriented vector orbibundles of rank~$2$. The tools obtained here are essential for calculating the Euler number of the disc orbibundles constructed in \cite{bgr} and provide a general framework that may also be applied to the examples obtained in \cite{discbundles} and \cite{sashanicolai}. At the core of these calculations and of the technology developed here lies Lemma \ref{circle section lemma} linking the geometry and the topology of $\SP^1$-orbibundles.
	
	At first glance, creating a framework for calculating the Euler number at the level of orbibundles may seem unnecessary since in the previous works \cite{discbundles} and \cite{sashanicolai} the authors were able to avoid such an approach. However, in their examples, the involved orbifolds $B$ were particularly simple and it was therefore possible to explicitly describe compact oriented surfaces covering them; through the relations between the fundamental domains of the orbifold and of such surfaces, the authors were able to calculate the Euler number by reducing the problem to the case of a disc bundle over surface. In~\cite{bgr}, on the other hand, this method does not work because there are no special explicit surfaces to rescue us (clearly, by Selberg's lemma, there exist finite covers by surfaces of the orbifolds used in~\cite{bgr}, but there is no practical way of determining them). Hence, we found it necessary to develop the technology presented in this article.
	
	Since the fibers of a complex hyperbolic disc bundle $M\to B$ over a surface are contractible, the fundamental groups of $B$ and $M$ are equal and $\pi_1(M)$ (viewed as a deck group) is a subgroup of the orientation preserving isometry group $\PU(2,1)$ of $\HH_\CC^2$ because $M$ is uniformized by $\HH_\CC^2$. Therefore, we have a discrete faithful representation $\rho:\pi_1(B) \to \PU(2,1)$ and $M$ is isometric to~$\HH_\CC^2/\pi_1(B)$. Isometric complex hyperbolic disc bundles correspond to conjugated representations. With that in mind, we can view complex hyperbolic disc bundles (up to isometry of bundles) over the surface $B$ as points in the $\PU(2,1)$-character variety of $B$ (that is, the space of all $\PU(2,1)$-representations of $\pi_1(B)$ modulo conjugation). It is worthwhile mentioning that character varieties are the essential geometrical objects in Higher Teichmüller theory and we believe that the methods developed here can be of use in some settings of that science as well.
	
	It is also important to point out that this work led us to a perspective shift. Originally, our objective in \cite{bgr} was to produce disc bundles over surfaces with complex hyperbolic structures as in \cite{discbundles} and \cite{sashanicolai}. Nevertheless, it is actually not desirable to reduce the examples found at the orbifold level to the surface level by pulling back the disc orbibundle to a disc bundle over a surface. By doing that, we lose track of the $\PU(2,1)$-character variety we are dealing with, thus losing information (for instance, rigid representations may become flexible). The techniques developed here allow us to deal with orbibundles on their own.
	
	To work with complex hyperbolic geometry on orbibundles, we slightly generalize the usual definition of orbibundle (see Subsection \ref{subsection: orbibundles}) and introduce the concept of an {\bf orbigoodle} (see Definition \ref{definition: orbigoodle}), a natural class of orbibundles over good orbifolds. The latter concept is well-behaved under pullbacks and enables a Chern-Weil theory for vector orbigoodles (see Theorem~\ref{euler thm} and Section \ref{chernweil}). Analytical expressions for the Euler number of rank $2$ vector orbigoodles via Chern-Weil theory are given in Theorems \ref{chern=euler2} and \ref{chern=euler}.
	
	Besides the Euler number, there is another important invariant associated to a complex hyperbolic disc orbigoodle $M\to B$, called the {\bf Toledo invariant} (see Definition \ref{toledo invariant}). It is a real-valued function $\tau$ defined on the $\PU(2,1)$-character variety of $B$ (see Definition \ref{character variety}). As the relative Euler number, the {\bf relative Toledo invariant} given by $\tau_R:=\tau/\chi(B)$ is also preserved under finite covers of $B$. When $B$ is a surface, D. Toledo proved that $|\tau_R|\leq 1$ as well as the famous Toledo rigidity: a representation $\rho:\pi_1(B) \to\PU(2,1)$ is maximal, i.e., the identity $|\tau_R(\rho)|= 1$ holds, if, and only if, there is a complex geodesic (see Subsection \ref{subsection complex hyperbolic geometry}) in $\HH_\CC^2$ stable under action of $\rho$. We give a proof of the Toledo rigidity for $2$-orbifolds (see Theorem \ref{toledo rigidity}). It is interesting to point out that the maximality of the Toledo invariant of a representation $\rho:\pi_1^\orb(B) \to \PU(2,1)$ implies that the representation $\rho$ is discrete \cite[Theorem 4.1]{biw}.
	
	We also show that the Toledo invariant of a representation $\rho:\pi_1^\orb(B) \to \PU(2,1)$ belongs to~$\frac 23\big(\ZZ+\frac1{m_1}\ZZ+ \cdots +\frac1{m_n}\ZZ\big)$
	for a $2$-orbifold $B$ whose singularities are conic points of angles $2\pi/m_1,\ldots ,2\pi/m_n$ (see Corollary \ref{discretetau}). When $B$ is a surface we recover the integrality property proved in \cite{GKL}. Note that this result implies the discreteness of the Toledo invariant. For surfaces, the Toledo invariant indexes the connected components of the character variety (see \cite[Theorem~1.1]{xia}): different connected components have different Toledo invariants. We conjecture the same for $2$-orbifolds.
	
	All examples found in \cite{discbundles}, \cite{sashanicolai}, and \cite{bgr} satisfy
	\begin{equation}\label{holomorphic section identity}
	\frac{3}{2}\tau_R(\rho) = e_R(M)+1,
	\end{equation}
	where $\rho$ is the representation associated to a complex hyperbolic disc bundle $M\to B$. The other non-trivial known examples of complex hyperbolic disc bundles are the ones found in \cite{GKL}, which satisfy $\frac 32 \tau_R(\rho)\leq e_R(M)+1.$
	
	It is known (see \cite{discbundles} or Corollary \ref{holomorphicsectionidentity}) that the existence of a holomorphic section of a complex hyperbolic disc bundle implies the identity \eqref{holomorphic section identity}. By holomorphic section (when $B$ is a surface) we mean a section $\sigma:B \to M$ that, viewed as a submanifold of $M$ (which is a complex manifold), is a Riemann surface. The existence of holomorphic sections for the examples in \cite{discbundles} was proved by Misha Kapovich (see \cite[Example 8.9]{kap3}). Nevertheless, his technique does not seem to work for the examples in \cite{bgr} because the complex hyperbolic disc orbigoodles we found are not rigid (they form $2$-dimensional regions in the corresponding character variety which means these orbigoodles, while equal as smooth orbigoodles, are geometrically distinct). Inspired by Toledo rigidity and based on the above observations, Carlos Grossi suggests the following conjecture: Given a complex hyperbolic disc bundle $M \to B$ over a compact oriented surface with negative Euler characteristic, we have
	$$\frac{3}{2}\tau_R(\rho) \leq e_R(M)+1$$
	and the equality holds if, and only if, there is a holomorphic disc embedded in $\HH_\CC^2$ stable under action of $\rho$. 
	
	The theory of orbibundles is well-known in geometry and is usually approached from the Lie groupoid perspective (see \cite{alr} and \cite{ame}). Nevertheless, we won't follow this path for two reasons: 1) the way we empirically discovered the formula for the Euler number of the orbibundles in \cite{bgr} does not have an algebraic flavor; it is actually low-tech and very geometric in nature, not fitting the language of Lie groupoids and 2) we plan to generalize the technology developed here for spaces much more singular than orbibundles because we believe there is a correspondence between the faithful part of the $\PU(2,1)$-character variety of hyperbolic spheres with $3$ conic points $B$ (see Figure \ref{pinchedorbifold}) and some very singular ``bundles''. Roughly speaking, the (wild) speculations are the following: for each faithful representation $\rho$ in the $\PU(2,1)$-character variety of $B$, we can find a ``polyhedron'' $Q\subset\HH_\CC^2$ (analogous to the actual polyhedra appearing as fundamental domains for disc orbigoodles in \cite{bgr}). The representation $\rho$ provides gluing relations between the sides of $Q$ and by taking the quotient of $Q$ by these relations we obtain a diffeological space $M$. The space~$Q$ can be seen as an immersion of $D\times P$ in $\HH_\CC^2$, where $D$ is a disc and $P$ is a fundamental domain for the hyperbolic sphere with three conic points. By gluing the sides of $Q$ we also glue the sides of the immersed fundamental domain $P$, obtaining a ``pinched orbifold'' (see Figure \ref{pinchedorbifold}). We believe the quotient $M$ to be, in some sense, a ``disc bundle'' over this pinched orbifold. Our suspicions come from computational observations: the formula for the Euler number in \cite{bgr} is well-defined and the identity \eqref{holomorphic section identity} holds even when the representation $\rho$ is no longer discrete. With these strange bundles in mind, it seemed appropriate to approach the subject through diffeology, a differential geometry like science that deals with very singular spaces.
	\begin{figure}[H]
		\centering
		\begin{minipage}{.5\textwidth}
			\centering
			\includegraphics[scale = .5]{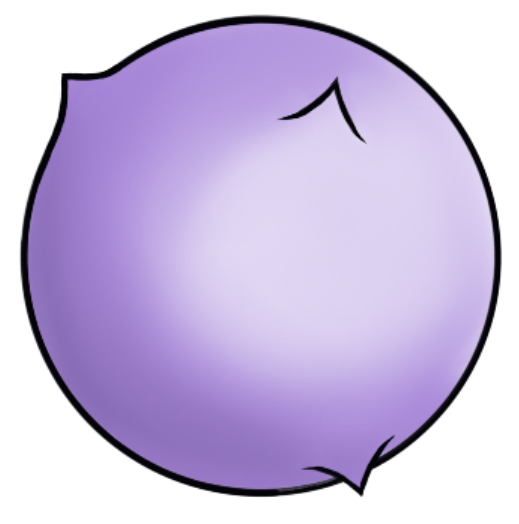}
		\end{minipage}%
		\begin{minipage}{.5\textwidth}
			\centering
			\includegraphics[scale =0.5]{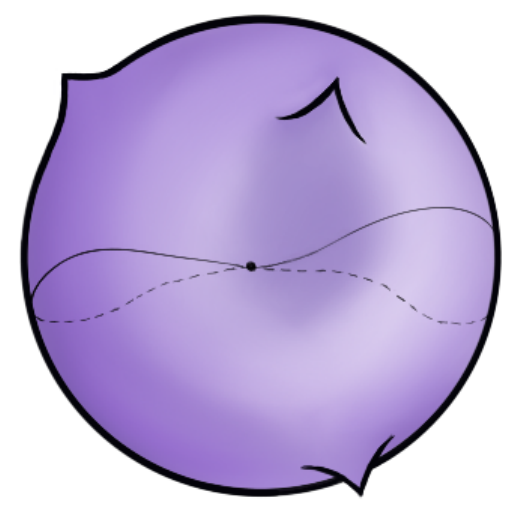}
		\end{minipage}
		\caption{ Orbifold and its pinched version.}
		\label{pinchedorbifold}
	\end{figure}

	 {\bf Acknowledgments:}	Thank you Carlos H.~Grossi and Sasha Anan'in for converting me to geometry.

	\section{Preliminaries}
	\subsection{Orbifolds as diffeological spaces}
	Orbifolds are the simplest type of singular spaces. They are almost like smooth manifolds, but allow singularities. In the case we are interested in (compact and oriented two-dimensional orbifold with only isolated singularities), orbifolds are, vaguely speaking, topological surfaces that from the smooth perspective are locally Euclidean except around the singular points, where the smooth structure is modeled by cones. The classical approach to orbifolds can be found, for instance, in \cite{sco}, \cite[Chapter 6]{kap2}, and \cite[Chapter~13]{thu}.

	In order to develop the theory of ``bundles'' over orbifolds we are interested in, the language of diffeology will come in handy. Since this language is not commonly used by hyperbolic geometers, we state some basic definitions. For a more complete treatment, see \cite{igl1} and \cite{igl2}.
	
	We will call an open set $U\subset\RR^n$ an Euclidean open set (note that $\RR^0=0$).
	
	\begin{defi}
		Let $M$ be a set. A {\bf diffeology} on $M$ is a family $\mathcal F$ of functions from Euclidean open sets to $M$ such that:
		\begin{itemize}
			\item every map $0 \to M$ belongs to $\mathcal F$;
			\item if $\phi:U \to M$ belongs to $\mathcal F$ and $g:V \to U$ is smooth, where $V$ is an Euclidean open set, then $\phi \circ g \in \mathcal F$;
			\item if we have a function $\phi:U \to M$ and every point $x\in U$ has a neighborhood $U_x \subset U$ such that $\phi|_{U_x} \in \mathcal F$, then $\phi\in \mathcal F$.
	\end{itemize} 
	
		The pair $(M,\mathcal F)$ is a {\bf diffeological space}. The elements of $\mathcal F$ are called {\bf plots} of $M$.
	\end{defi}
	Structurally, the definition of a diffeology resembles that of a topology. In particular, every set $M$ admits two extreme diffeologies:
	
	\begin{itemize}
		\item The {\bf discrete diffeology}: the set of all functions from Euclidean open sets to M that are locally constant.
		\item The {\bf indiscrete diffeology}: the set of all functions from Euclidean open sets to M. 
	\end{itemize} 
	
	A map $f:M_1 \to M_2$ is {\bf smooth} if, for every plot $\phi:U \to M_1$ of $M_1$, the map $f \circ \phi: U \to M_2$ is a plot of $M_2$. In particular, plots are smooth. Therefore, we have the category of diffeological spaces where the objects are diffeological spaces and the morphisms are the smooth maps. A diffeomorphism is an isomorphism in this category. 
	
	The basic constructions we will need are motivated by the ideas of {\bf initial and final diffeologies} (analogous to the concepts of initial and final topologies). Given a set $M$ and a family of diffeological spaces $M_\alpha$ with maps $\pi_\alpha:M \to M_\alpha$, the initial diffeology on $M$ is the largest one that turns each $\pi_\alpha$ into a smooth map. The dual definition is the final diffeology, i.e., given a set $M$ and a family of diffeological spaces $M_\alpha$ with maps $i_\alpha: M_\alpha \to M$, the final diffeology of $M$ is the smallest one such that each $i_\alpha$ is smooth.
	
	Some basic constructions arising from initial diffeologies include the subspace diffeology and the product diffeology, given respectively by the inclusion and the projection maps. Final diffeologies provide, for example, the quotient and the coproduct diffeologies, given respectively by the quotient map and the natural inclusions in the disjoint union.
 
    The natural topology on a diffeological space $M$, known as the D-topology, is the largest topology that renders all plots of $M$ continuous.

	\begin{rmk} The category of (finite-dimensional) smooth manifolds does not have arbitrary pullbacks (transversality is required) and quotients. Furthermore, the space of smooth maps between manifolds as well as diffeomorphism groups are not smooth manifolds. On the other hand, the category of diffeological spaces is bicomplete and it is Cartesian closed; in particular, the space of smooth maps between diffeological spaces is a diffeological space, and diffeomorphism groups are diffeological groups. Moreover, the theories of bundles and differential forms are very simple and practical in this language.
	\end{rmk}
	
	With these basics definitions, we can define an $n$-dimensional smooth manifold as a Hausdorff and second countable diffeological space locally diffeomorphic to open subsets of $\RR^n$, i.e., given $x \in M$, there is a neighborhood $D$ of $x$ in the diffeological topology of $M$, such that $D$ is diffeomorphic to an open subset of $\RR^n$. Similarly, we have orbifolds:
	
	\begin{defi}\label{defiorbifold} An {\bf$n$-dimensional orbifold $B$} is a Hausdorff and second countable diffeological space locally diffeomorphic to $\BB^n/\Gamma$, where $\BB^n \subset \RR^n$ is the unit open $n$-ball centered at $0$ and $\Gamma$ is a finite subgroup of the orthogonal group $\OO(n)$. We call such diffeomorphism $\phi:\BB^n/\Gamma \to D$, where $D$ is an open subset of $M$, an orbifold chart and say the chart $\phi$ is centered at $x$ if $\phi\big([0]\big) =x$. An orbifold is {\bf locally orientable} if it is locally modeled on $\BB^n/\Gamma$ with $\Gamma\subset \SO(n)$. 
	\end{defi}
	For a detailed treatment of orbifolds as diffeological spaces see \cite{IKZ}.
	 
	All orbifolds in this article are supposed to be locally orientable because we do not allow non-isolated singularities in two-dimensional orbifolds.
	
	\begin{defi} Consider an orbifold $B$. We say $x \in B$ is a {\bf singular point} if there is a chart $\phi:\BB^n/\Gamma \to D$ centered at $x$ and $\Gamma \neq \{1\}$. If a point is non-singular, we call it regular. Since we are dealing with locally orientable orbifolds, $B$ is connected if, and only if, $B\setminus S$ is connected, where $S$ is the set of all singular points of $B$. We say $B$ is an {\bf oriented orbifold} if $B\setminus S$ is oriented. 
	\end{defi}

	\begin{rmk} We will need the concepts of orbifold fundamental group and orbifold universal cover. For details, the reader may take a look at \cite[Chapter 6]{kap2} and \cite[Chapter 13]{thu}. We will also use the concept of a good orbifold, which is an orbifold covered (in the sense of orbifolds) by manifolds. In this case, an orbifold $B$ can be seen as the quotient of a simply connected manifold $\HH$ by a group $G$ acting properly discontinuously on $\HH$, and $\pi_1^\orb(B) = G$.
	\end{rmk}

	\subsection{Complex hyperbolic geometry}\label{subsection complex hyperbolic geometry}
	In this subsection, we present a few basic facts about complex hyperbolic geometry that will be needed later. The reader is referred to \cite{coordinatefree} and \cite{goldmanbook} for details.
	
	Consider an $(n+1)$-dimensional complex vector space $V$ endowed with a Hermitian form $\langle \cdot,\cdot \rangle$ of signature $-+\cdots+$. The projective space $\PP(V)$ is divided into three parts, 
	$$
	\HH_\CC^n = \{\pp \in \PP(V): \langle p,p\rangle <0\}, \quad \mathrm{S}(V) =\{ \pp \in \PP(V): \langle p,p \rangle =0 \}, 
	$$
	$$\mathrm E(V) = \{\pp \in \PP(V): \langle p,p \rangle>0\},$$
	where we denote by $\pp$ a point in the projective space and by $p$ a representative of $\pp$ in $V$. The $n$-dimensional ball $\HH_\CC^n$ is called {\bf complex hyperbolic space}. We say the points of $\HH_\CC^n$, $\mathrm{S}(V)$ and~$\mathrm E(V)$ are negative, isotropic, and positive, respectively.
	  
	The tangent space of  $\TT_{\pp} \PP(V)$ is naturally identified with the space of linear transformations from $\CC p$ to $p^\perp:=\big\{v\in V: \langle p,v\rangle=0\big\}$  whenever $\pp$ is non-isotropic. Furthermore, over the non-isotropic region we define the {\bf Hermitian metric} $$h(s,t): = - \frac{\langle s(p),t(p)\rangle}{\langle p,p \rangle}$$
	for $t,s \in \TT_{\pp} \PP(V)$ (the definition does not depend on the representatives for $\pp$ because $t,s$ are~linear). The {\bf pseudo-Riemannian metric} $g$ and the {\bf Kähler form} $\omega$ are defined by $g:=\real h$ and $\omega:=\imag h$, respectively. Furthermore, the metric $g$ is complete and it is Riemannian on $\HH_\CC^n$. Hence, $\HH_\CC^n$ is a Kähler manifold.
	
	The {\bf Riemann curvature tensor} $R$ of the Levi-Civita connection is
	\begin{equation}\label{riemanntensor}R_{\pp}(t_1,t_2)s=h_{\pp}(t_2,t_1)s+h_{\pp}(s,t_1)t_2-h_{\pp}(t_1,t_2)s - h_{\pp}(s,t_2)t_1
	\end{equation}
	for $t_1,t_2,s \in T_{\pp} \PP_(V)$.
	
	\begin{rmk} We use the definition $R(X,Y)Z := \nabla_X\nabla_Y Z - \nabla_Y\nabla_X Z- \nabla_{[X,Y]}Z$, which differs by a sign from the one in \cite{discbundles}.
	\end{rmk}
	
	From the above formula, it is easy to deduce that the Gaussian curvature of $\HH_\CC^1$ is $-4$, hence it is a Poincaré disc. Additionally, $E(V)$ is also a Poincaré disc with the same curvature. So the projective line $\PP_\CC^1$ is formed by two Poincaré discs glued along the boundary, and we call it a {\bf Riemann-Poincaré sphere}. 
	
	The {\bf complex hyperbolic plane} $\HH_\CC^2$ has non-constant sectional curvature varying on the interval $[-4,-1]$. Furthermore, given a positive point $\pp$, the projective line $\PP(p^\perp)$ is a Riemann-Poincaré sphere, and its intersection with $\HH_\CC^2$ is a {\bf complex geodesic}. Note that two distinct points in the complex hyperbolic plane determine a unique complex geodesic.
	
	It may seem strange that we do not scale the metric in order to obtain complex geodesics of curvature $-1$. The reason is that there exist Beltrami-Klein hyperbolic discs inside $\HH_\CC^2$ as well. They are given by $\HH_\CC^2\cap\PP(\RR e_1 \oplus \RR e_2 \oplus \RR e_3)$, where $e_1,e_2,e_3$ is an orthonormal basis of $V$. These discs have curvature $-1$, are called {\bf $\RR$-planes}, and are very different from complex geodesics: they are not uniquely determined by two points, they are not embedded Riemann surfaces, and are Lagrangian submanifolds (whereas complex geodesics are symplectic).
		
	The group of holomorphic isometries of $\HH_\CC^n$ is the group $\PU(n,1)$, the projectivization of the group $\SU(n,1)$ formed by the linear isomorphisms of $V$ preserving the Hermitian form.
	
	A non-identical isometry that fixes a point in $\HH_\CC^n$ is called an {\bf elliptic isometry}. An elliptic isometry can have a unique fixed point in $\HH_\CC^2$ or a fixed complex geodesic: if $\widehat I \in \SU(2,1)$ is a representative of an elliptic isometry $I$ (the representatives differ by a cube root of the unit), then~$\widehat I$ has an orthonormal basis $c,p,q$ consisting of eigenvectors, where $c$ is negative. The projective lines $\PP(\CC c \oplus \CC p)$, $\PP(\CC c \oplus \CC q)$, and $\PP(\CC p \oplus \CC q)$ are stable under the action of $I$, where the first two are Riemann-Poincaré spheres and the third is a round sphere (of positive points). If two of the eigenvectors have the same eigenvalue, then the projective line determined by them is fixed. So, an elliptic isometry has either a unique fixed point in $\HH_\CC^2$ or a unique fixed complex geodesic.
	
	\section{Orbibundles and the Euler number}
	\subsection{Orbibundles}\label{subsection: orbibundles}
    Let $M,F$ be diffeological spaces and let $B$ be an orbifold.  A smooth map $\zeta:M\to B$ is an {\bf orbibundle}  with fiber $F$ if for every point $p \in B$ there is an orbifold chart $\phi:\BB^n/\Gamma \to D$ centered at $p$ satisfying the following properties:
    \begin{itemize}
        \item there is a smooth action of $\Gamma$ on $\BB^n \times F$ of the form $h(x,f) = (hx,a(h,x)f)$, where\break $a: \Gamma \times \BB^n \to \mathrm{Diff}(F)$ is smooth and $\mathrm{Diff}(F)$ stands for the group of diffeomorphisms of~$F$ endowed with its natural diffeology (see \cite[Section 1.61]{igl2});
        \item there is a diffeomorphism $\Phi:(\BB^n \times  F)/\Gamma \to \zeta^{-1}(D)$ such that the diagram

		\begin{equation}\label{definition orbibundle}
		\begin{tikzcd}
		(\BB^n \times F)/\Gamma \arrow[rr, "\Phi"] \arrow[d,"\mathrm{pr}_1"'] &  & \zeta^{-1}(D) \arrow[d, "\zeta"] \\
		\BB^n/\Gamma \arrow[rr, "\phi"]                                &  & D                               
		\end{tikzcd}
		\end{equation}
	    commutes, where $\mathrm{pr}_1([x,f])=[x]$.
	   \end{itemize}  
	   
		Note that a fiber bundle in the sense of \cite[Chapter 8]{igl2} (when the base space is an orbifold) is the particular case of an orbibundle where $a(h,x)f$ always equals $f$. If $F=\BB^2$, then we say that $\zeta$ is a {\bf disc orbibundle}.
	
	\begin{rmk}
	Our definition was originally discovered by studying examples in complex hyperbolic geometry and is heavily inspired by Audin's approach to Seifert manifolds \cite{audin}. We later found out a similar definition in \cite[3.3 Orbibundles and Frobenius’ Theorem]{car} which does not use diffeology and assumes the total space to be an orbifold. Not requiring the total space to be an orbifold makes the description of a $G$-orbibundle (see Definition \ref{gorbibundle}) more natural. Moreover, the language of diffeological spaces simplifies the theory due to its good categorical properties (a highlight being the fact that the category has quotients).
	\end{rmk}
	
	If $F$ is a discrete (and countable) diffeological space, we say that $\zeta:M\to B$ is an {\bf orbifold covering map} and if $F$ is finite with $d$ elements we say the orbifold cover has degree $d$. Let us analyze more precisely what this object is. Following the above diagram, $\zeta^{-1}(D)$ is modeled by $(\BB^n \times F)/\Gamma$. 
    Note that $\Gamma$ acts directly on $F$ because $F$ is discrete. 
    Writing $F/\Gamma = \{\Gamma f_1,\ldots, \Gamma f_l\}$, the set of the disjoint orbits of $F$,  we have the diffeomorphism
	\begin{align*}
	\coprod_{i=1}^l \mathbb B^n/\mathrm{stab}_\Gamma (f_i) &\to (\mathbb B^n \times F)/\Gamma 
	\\{([x],f_i)}            & \mapsto {[x,f_i]}                 
	\end{align*}
	Therefore, $M$ is an orbifold.
	\begin{defi} An orbibundle $\zeta:M \to B$ with fiber $\RR^k$ is a {\bf real vector orbibundle of rank~$k$}~if:
	\begin{itemize}
		\item every fiber of $\zeta$ over a regular point is a real vector space,
		\item for each orbifold chart $\phi:\BB^n/\Gamma \to D$, following the notation in diagram \eqref{definition orbibundle}, the map $a(h,x)$ is a linear isomorphism; 
		\item for each regular point $u \in \BB^n/\Gamma$, the map $\Psi:\mathrm{pr}_1^{-1}(u) \to \zeta^{-1}(\phi(u))$ is an $\RR$-linear isomorphism.
	\end{itemize}
	\end{defi}
	Complex vector orbibundles are defined analogously. Furthermore, we say $M \to B$ is an {\bf oriented vector orbibundle} if removing the singular points of $B$ we obtain an oriented vector bundle.

	\begin{rmk} For each $u \in \BB^n$ the map $a(\cdot,u):\stab_{\Gamma}(u) \to \mathrm{GL}(k,\RR)$ is a group action and $\mathrm{pr}_1^{-1}(u) \simeq \RR^k/\stab_{\Gamma}(u)$. In particular, whenever $u$ is a regular point, $\mathrm{pr}_1^{-1}(u)$ is a linear space.\end{rmk}
	\begin{defi} \label{gorbibundle}
	Consider a diffeological group $G$. We say an orbibundle  $\zeta:M \to B$ with fiber $G$ is a {\bf$G$-orbibundle} if
	
	\begin{itemize}
		\item $G$ acts smoothly on $M$;
		\item for each orbifold chart $\phi:\BB^n/\Gamma \to D$, according to the notation in diagram  \eqref{definition orbibundle}, we consider $a: \Gamma \times \BB^n\to G$ and the action given by $h(x,g) := (hx,a(h,x)g)$;
		\item the map $\Phi: (\BB^n\times G)/\Gamma \to \zeta^{-1}(D)$ is $G$-equivariant, where $G$ acts on $(\BB^n\times G)/\Gamma$ on the right.
	\end{itemize}
	\end{defi}
	
	In this paper, $\SP^1$ stands for the group of unit complex numbers.
	
	\begin{prop} For\/ $\SP^1$-orbibundles, up to an $\SP^1$-equivariant diffeomorphism, we can assume that the map $a:\Gamma\times\BB^n\to \SP^1$ does not depend on $x$.
	\end{prop}
	\begin{proof} 
	Given a small ball $U \subset \BB^n$ centered at $0$, define $f:U \times \SP^1 \to U \times \SP^1$,
	$$f(x,s) :=\Big(x, \mathrm{norm}\big(\sum_{h\in \Gamma}  a(h^{-1},0) a(h,x)\big)s\Big),$$
	where $\mathrm{norm}(z) := z/|z|$, $z \in \CC$. Define actions of $\Gamma$ on the domain by $h(x,s) := (hx,a(h,x)s)$ and on the codomain by $h(x,s) := (hx,a(h,0)s)$. Since $a(hg,x)=a(h,gx)a(g,x)$,
	we have
	\begin{align*}f(g(x,s)) &= \Big(gx, \mathrm{norm}\big(\sum_{h \in \Gamma}  a(h^{-1},0) a(h,gx) \big) a(g,x) s\Big)
	\\&=\Big(gx, \mathrm{norm}\big(\sum_{h \in \Gamma} a(h^{-1},0)a(hg,x)\big)s\Big)
	\\&=\Big(gx, \mathrm{norm}\big(\sum_{h \in \Gamma} a(gh^{-1},0) a(h,x) \big)s\Big)
	\\&=\Big(gx,a(g,0) \mathrm{norm}\big(\sum_{h \in \Gamma}  a(h^{-1},0)a(h,x)\big)s\Big)
	\\&=gf(x,s).
	\end{align*}
	Hence the map $f$ is a $\Gamma$-equivariant diffeomorphism and 
	$$f:(U \times \SP^1)/\Gamma \to (U \times \SP^1)/\Gamma$$
	is a diffeomorphism. Since the action of $\SP^1$ on both quotients is on the right, we conclude it is an $\SP^1$-equivariant diffeomorphism.
	\end{proof}
	
	\begin{rmk} The above argument can be easily adapted to the case of ${\rm GL}_n$ and ${\rm SL}_n$-orbibundles.
	
	\end{rmk}
	
	\subsection{Euler number of $\SP^1$-orbibundles over $2$-orbifolds}  \label{subsection eulernumber}
	In this subsection, we establish some tools to calculate Euler numbers and to characterize $\SP^1$-orbibundles by such invariant.

	Consider a $2$-orbifold $B$, a point $x\in B$, and an orbifold chart $\phi:\BB^2/\Gamma \to B$ centered at~$x$. Note that $\Gamma$ is a finite subgroup of $\SO(2)$ and, therefore, our orbifolds have only isolated singularities since $\Gamma$ acts freely in $\BB^2 \setminus \{0\}$.
	
	Thinking of $\BB^2$ as the unit disc in $\CC$ centered at $0$, we have $\Gamma=\langle \xi \rangle$, where $\xi = \exp(2\pi i/n)$. The number $n$ is the {\bf order} of the point $x$. If $n=1$, then $\Gamma={1}$ and the point is regular. In the case $n\geq 2$, the quotient $\BB^2/\Gamma$ is a cone with cone point of angle $2\pi/n$.
	
	The orbifolds we are interested in are compact and, therefore, have a finite number of singular points. In this section, when we consider a $2$-orbifold $B$, we assume that it is connected, compact, and oriented. Moreover, $x_1,\ldots, x_n$ denote its singular points.
	
	\begin{defi}[Euler number of $\SP^1$-orbibundles]\label{definition: euler number}
	Let $\zeta:M\to B$ be an 
    oriented
    $\SP^1$-orbibundle and let $x_0\in B$ be a regular point. For each $x_i$, $i=0,1,\ldots,n$, consider a small smooth disc $D_i\subset B$ centered at $x_i$. We have the surface with boundary $B' = B \setminus \sqcup_k D_k$. Note that $\zeta:M' \to B'$, where $M':=\zeta^{-1}(B')$, is an ordinary $\SP^1$-bundle and $B'$ is homotopically equivalent to a graph. Therefore, $\zeta:M' \to B'$ admits a global section $\sigma$ because $\SP^1$-bundles over graphs are trivial. 
    Fix an arbitrary fiber $s$ of $\zeta:M \to B$ over a regular point
    (every time we talk about an $\SP^1$-fiber we will assume that it is parameterized and that the curve goes around the fiber just once in the same direction of the $\SP^1$-orbits). 
    We define the {\bf Euler number} $e(M)$ by the formula 
	$$
	\sigma|_{\partial B'} = -e(M) s 
	$$
	in $H_1(M,\QQ)$.
	\end{defi}
	
	Let us prove that $e(M)$ does not depend on the choice of $\sigma$. Given sections $\sigma_1$ and $\sigma_2$ of $\zeta:M' \to B'$ we define $f:B' \to \SP^1$ such that $\sigma_1(x)=f(x)\sigma_2(x)$, $x\in B'$.
	Note that in homology $\sigma_1|_{\partial D_k} = \deg(f|_{\partial D_k})s+\sigma_2|_{\partial D_k}$ and, therefore,
	$$\sum_k \sigma_1|_{\partial D_k} - \sum_k \sigma_2|_{\partial D_k} = \sum_k \deg(f|_{\partial D_k})s.$$
	On the other hand, if $d\theta:=-idz/z$ is the angle $1$-form of $\SP^1$, then, using Stokes theorem and the fact that the orientation of $\partial D_k$ is opposite to that of $\partial B'$, we obtain
	$$\sum_k \deg(f|_{\partial D_k}) =\sum_k \frac{1}{2\pi}\int_{\partial D_k} f^\ast d\theta = - \frac{1}{2\pi}\int_{ B'} d(f^\ast d\theta) = 0 ,$$
	$$
	\sigma_1|_{\partial B'} = -\sum_k\sigma_1|_{\partial D_k} = -\sum_k \sigma_2|_{\partial D_k} = \sigma_2|_{\partial B'}.
	$$
	Hence, the Euler number is well-defined. It is easy to verify that the definition of the Euler number does not depend on the choice of the regular point $x_0$.

	Next, we define the Euler number of oriented rank $2$ real vector orbibundles. Recall that by oriented here we mean that, removing the singular points of the base orbifold, one obtains an oriented vector bundle over a surface.
	
	As a motivation, consider an oriented real vector bundle $\zeta:L \to B$ of rank $2$ over a connected, compact, and oriented surface. If $X$ is a section transversal to the $0$ section in $L$, then these sections intersect in a finite number of points whose projections onto $B$ are denoted by $x_1',\ldots,x_n'$. On $L$ we consider local trivializations $\xi_x:L|_{\overline D_x}\to \overline D_x \times \RR^2$ such that $D_x$ is a small open disc centered in $x$ with smooth boundary and this atlas of $L$ is compatible with the orientation of the bundle. Consider as well the surface with boundary $B':=B\setminus \sqcup_i D_{x_i}$. 
	
	By the Poincaré-Hopf theorem, the Euler number of the vector bundle $\zeta$ is given by the formula
	\begin{equation}\label{eq:poincare hopf}
	e(M)=\sum_{k=1}^n \textrm{index}(X,x_k').
	\end{equation}
	
	\begin{rmk} In a finite-dimensional real vector space $V$ consider the sphere $\SP(V) := V^\times/\RR_{>0}$, where $V^\times := V\setminus 0$. The topology and smooth structure of $\SP(V)$ can be induced from a sphere obtained from an inner product in $V$. Furthermore, these two structures do not depend on the choice of the inner product. Analogously, given a real vector orbibundle $L \to B$ we can define the natural sphere orbibundle $\SP(L) \to B$. In the case of 
    an oriented
    real vector orbibundle of rank $2$, it is always possible to introduce an $\SP^1$-action compatible with the natural orientation of the fibers (just consider a metric on the vector orbibundle); the Euler number will not depend on the choice of the action.
	\end{rmk} 
	
	The term $\textrm{index}(X,x_k')$ in equation \eqref{eq:poincare hopf} is the degree of the map $\partial D_{x_k'} \to \SP(\RR^2)$ between circles given by $x \mapsto \mathrm{pr}_2\, \xi_{x_k'}\big(X(x)\big)$, where $\mathrm{pr}_2:D_{x_k'} \times ({\RR^2})^\times  \to \SP(\RR^2)$ is the projection on the second coordinate composed with the quotient map $(\RR^2)^\times \to \SP(\RR^2)$. 

	Consider the section $\sigma(x):=[X(x)]$  of $\SP(L)|_{B'} \to B'$ and an oriented fiber $s$ of $\SP(L) \to B$. Observe that $\sigma|_{\partial D_{x_k}} \!\!= \textrm{index}(X,x_k') s$ in homology. Indeed, the oriented fiber over $x_k'$ generates the fundamental group of the solid torus $\zeta^{-1}(\overline{D_{x_k'}})$ and two oriented fibers are homotopic). Hence
	$$\sigma|_{\partial B'} = -e\big(\SP(L)\big)s \text{ in }H_1(\SP(L),\ZZ)$$
	because $\sigma|_{\partial B'} = - \sum_k \sigma|_{\partial D_k}$ since the orientation of $\partial D_k$ is opposite to that of $\partial B'$.
	So, we can reduce our calculations to ordinary $\SP^1$-bundles.
	
	\begin{defi}
	Consider an oriented real vector orbibundle  $\zeta:L \to B$ of rank $2$, where $B$ is a connected, compact, and oriented  $2$-orbifold. We define the Euler number of $L$ to be the Euler number of $\SP(L) \to B$.
	\end{defi}
	
	An example of oriented $2$-dimensional real vector orbibundle is the tangent bundle $TB$ of $B$. For each $x\in B$ we have the stalk $C_x^\infty$ of the sheaf $C^\infty(-,\RR)$ and the space $\mathrm{Der_x}$ of derivations at~$x$, which is a real vector space. A tangent vector at a point $x\in B$ is a derivation $v$ that comes from a smooth curve  $\gamma:\RR \to M$ such that $\gamma(0)=x$, i.e., for every $f\in C_x^\infty$ we have $v(f):=(f\circ \gamma)'(0)$. The real vector space $\mathrm{T}_xB$ is defined as $\mathrm{Span}_\RR\{v \in \mathrm{Der_x}:v \text{ is a tangent vector at }x\}$. It is easy to verify that $\mathrm{T}B$ is a rank $2$ real vector orbibundle.
	
	\begin{lemma} \label{circle section lemma} Let\/ $\Gamma_1$ and\/ $\Gamma_2$ be finite subgroups of\/ $\SP^1$ such that\/ $\Gamma_2$ is a subgroup of\/ $\Gamma_1$. 
    Also consider the action\/ $\Gamma_1\times \SP^1 \to \SP^1$ given by $(\xi, \gamma) \mapsto \xi^k \gamma$, with $k \in \ZZ$, and the solid torus\/ $\TT =\overline{\BB}^2 \times \SP^1$.
    We have the following commutative diagram
		$$
		\begin{tikzcd}
		\TT/\Gamma_2 \arrow[r, "G"] \arrow[d, "\zeta_2"'] & \TT/\Gamma_1 \arrow[d, "\zeta_1"] \\
		\overline{\BB}^2/\Gamma_2 \arrow[r, "g"]                                    & \overline{\BB}^2/\Gamma_1                                  
		\end{tikzcd}    
		$$
		of natural continuous maps.
		
		With respect to the map $\zeta_i:\TT/\Gamma_i \to  \overline{\BB}^2/\Gamma_i$, consider the fiber $s_i:\SP^1 \to \zeta_i^{-1}([1])$ over $[1] \in \overline{\BB}^2/\Gamma_i$ given by $s_i(z) := [1,z]$, and the fiber $s_i':\SP^1/\Gamma_i \to \zeta_i^{-1}([0])$ over $[0]$ given by $s_i'([z]) := [0,z]$.
		
		Additionally, let $\sigma_i$ be a continuous section of the \/ $\SP^1$-principal bundle $\zeta_i:(\SP^1 \times \SP^1)/\Gamma_i \to \SP^1/\Gamma_i$, with $i=1,2$, such that $\sigma_2$ is on top of $\sigma_1$, i.e,
		$G\sigma_2 = \sigma_1g$. Then we have:
		\begin{itemize}
			\item $\TT/\Gamma_i$ is homeomorphic to $\TT$ and any fiber of $\zeta_i$ is a generator of $H_1\big(\TT/\Gamma_i,\QQ\big)$,    
			\item 
            $s_i=\frac{n_i}{d_i} s_i'$ in $H_1\big(\TT/\Gamma_i,\QQ\big)$,  where $n_i=|\Gamma_i|$ and $d_i=\gcd(k,n_i)$.
			\item  
            $[\Gamma_1:\Gamma_2]\sigma_1 = qs_1$ in $H_1\big(\TT/\Gamma_1,\QQ\big)$ if, and only if, $\sigma_2 = qs_2 $ in $H_1\big(\TT/\Gamma_2,\QQ\big)$.
		\end{itemize}
	\end{lemma}
	\begin{proof} 
	Let us prove the first item. 
    The action of $\Gamma_i$ on $\SP^1$ is $\Gamma_i \times \SP^1 \to \SP^1$, $(\xi,\gamma) \mapsto \xi^k\gamma$, where $\xi := \exp(2\pi i/n)$ and $n:=|\Gamma_i|$.
    The result is obvious for $k=0$. We assume $1 \leq k\leq n-1$. If $d:=\gcd(n,k)$, then there is an integer $l$ such that $kl =d \mod n$ and $1 \leq l\leq n-1$.  The continuous map $\lambda:\TT \to \TT$ given by $\lambda (z,\gamma) = (\gamma^{-l}z^d,\gamma^{n/d})$ is $\Gamma_i$-invariant and surjective. Besides, if $\lambda(z',\gamma') = \lambda(z,\gamma)$, then there exists $\omega \in \Gamma_i$ satisfying $(\omega z,\omega^k\gamma)=(z',\gamma')$. Indeed, there exists an integer $s$ such that $\gamma'=\xi^{sd} \gamma$, because $\gamma'^{n/d}=\gamma^{n/d}$ and, consequently, $z'^d = \xi^{lsd} z^d$. By the same reasoning, there is an integer $t$ satisfying $z' = \xi^{ls+(n/d) t} z$. Taking $\omega := \xi^{ls+(n/d) t}$ we have $\omega^k =\xi^{sd}$ and $(z',\gamma')=(\omega z,\omega^k\gamma)$.
		Therefore, we have the continuous bijection $\TT/\Gamma_i \to \TT$, $[z,\gamma] \mapsto \lambda(z,\gamma)$, which is a homeomorphism because $\TT/\Gamma_i$ is compact.
		
        The second item follows from the homotopy $R:[0,1] \times \SP^1 \to \TT/\Gamma_i$, $R(t,z):=[t,z]$, because $R(0,-) = \frac{n_i}{d_i} s_i'$ and $R(1,-) = s_i$ in $H_1\big(\TT/\Gamma_i,\QQ\big)$.

		Let us prove the third item. 
        At the homology level, $\sigma_1 \circ g = [\Gamma_1:\Gamma_2] \sigma_1$, since $g:\SP^1/\Gamma_2 \to \SP^1/\Gamma_1$ is a covering map of degree $[\Gamma_1:\Gamma_2]$. Additionally, $G_\ast (s_2) = s_1$ at the homology level. Therefore, if
		$\sigma_2 = q s_2$ then 
		$$G_\ast(\sigma_2) = q \,G_\ast(s_2),$$
		$$[\Gamma_1:\Gamma_2]\sigma_1 = qs_1$$ 
		in homology, due to $G\sigma_2 = \sigma_1g$.
		\end{proof}
	
	The Euler number of $\TT B$ is the {\bf Euler characteristic} of $B$, denoted by $\chi(B)$. 
	
	We prove the next theorem, which is a standard result in the theory of orbifolds, in order to illustrate a practical application of Lemma \ref{circle section lemma}. This may also be useful as the proofs of Theorems~\ref{thm classification circle bundles} and \ref{euler thm} follow similar lines of thought.
	
	\begin{thm}\label{eulerchar}The Euler characteristic of $B$ is given by
		$$
		\chi(B) = \chi(\tilde B) + \sum_{k=1}^n \bigg( -1 + \frac1{m_k}\bigg),
		$$
		where the numbers $m_1,\ldots,m_n$ are the orders of the singular points $x_1,\ldots,x_n$ of $B$ and $\tilde{B}$ stands for the smooth surface obtained by removing the singular discs and gluing regular ones.
	\end{thm}
	\begin{proof}
		Consider a regular point $x_0$ and remove from $B$ small open discs $D_k$ centered at $x_k$, where $k=0,\dots,n$, thus obtaining a surface with boundary $B'$. For each $1 \leq k \leq n$ consider a vector field $V_k$ on $\partial D_k$ pointing outwards $B'$. By  \cite[Lemma I.3.6]{audin}, there is a global section $\sigma$ of the $\SP^1$-bundle $\SP(\TT B')$ such that $\sigma|_{\partial D_k}(x) = [V_k(x)] \in \SP(\TT_x B)$ for $k=1,\ldots,n$. Let $s$ be a generic fiber of $\zeta:\SP(\TT B) \to B$ over a regular point, and let $s_k$ be the fiber over $x_k$. By Lemma \ref{circle section lemma} we have $s=m_k s_k$ in $H_1(M,\QQ)$.
		
		We can assume that for each $k$ there is a chart $\phi_k:2\BB^2/\Gamma_k \to 2D_k$ centered at $x_k$ satisfying $\phi_k\big(\overline{\BB}^2/\Gamma_k\big)=\overline{D}_k$, where $\Gamma_k=\langle \exp(2\pi i/m_k) \rangle$, $2\BB^2$ is the disc of radius $2$ in the complex plane, and $2D_k$ is an open subset of $B$ containing $D_k$. In particular, we have the diffeomorphism $\phi_k:\SP^1/\Gamma_k \to \partial D_k.$ Furthermore, we have a trivialization of the vector bundle $\TT B$ given by the chart $\phi_k$ as described in the commutative diagram
		$$
		\begin{tikzcd}
		(2\BB^2\times \RR^2)/\Gamma_k \arrow[r, "\Phi_k"] \arrow[d] & TB|_{2D_k} \arrow[d, "\zeta"] \\
		2\BB^2/\Gamma_k \arrow[r, "\phi_k"]                       & 2D_k                         
		\end{tikzcd}
		$$
		In particular, we have the commutative diagram
		$$
		\begin{tikzcd}
		\overline{\BB}^2\times \SP^1 \arrow[r] \arrow[d] & (\overline{\BB}^2\times \SP^1)/\Gamma_k \arrow[r, "\Phi_k"] \arrow[d] & \SP(\TT B)|_{\overline{D}_k} \arrow[d, "\zeta"] \\
		\overline{\BB}^2 \arrow[r]                       & \overline{\BB}^2/\Gamma_k \arrow[r, "\phi_k"]                         & \overline{D}_k                              
		\end{tikzcd}
		$$
		
		Now the section $\sigma|_{\partial D_k}: \partial D_k \to \SP(\TT B)|_{\partial D_k}$ can be lifted to a section $\tilde \sigma_k:\SP^1 \to \SP^1 \times \SP^1$ satisfying $\Phi_k([\tilde \sigma_k(z)])= \sigma|_{\partial D_k} \circ \phi_k([z])$, as in Figure \ref{liftingvectorfield}. Since the $\SP^1$-bundle $\overline{\BB}^2\times \SP^1 \to \overline{\BB}^2$ is trivial, $\tilde \sigma_k$ equals the fiber $0\times \SP^1$ in $H_1(\overline{\BB}^2\times \SP^1,\QQ)$. Then, by Lemma \ref{circle section lemma}, we have
		$$
		\sigma|_{\partial D_k} = s_k'
		$$
		in $H_1\big(\SP(TB)|_{\overline{B}_k},\QQ\big)$.
		\begin{figure}[H]
			\centering
			\begin{minipage}{.5\textwidth}
				\centering
				\includegraphics[scale = .3]{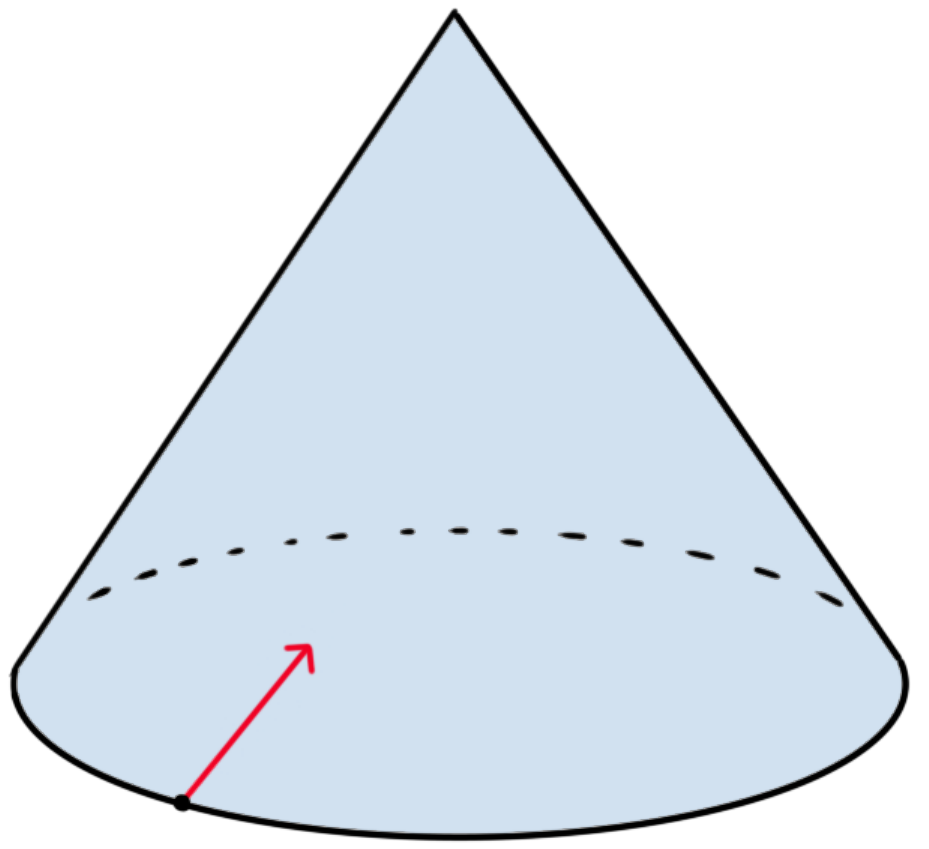}
				\caption*{(a)}
			\end{minipage}%
			\begin{minipage}{.5\textwidth}
				\centering
				\includegraphics[scale = .3]{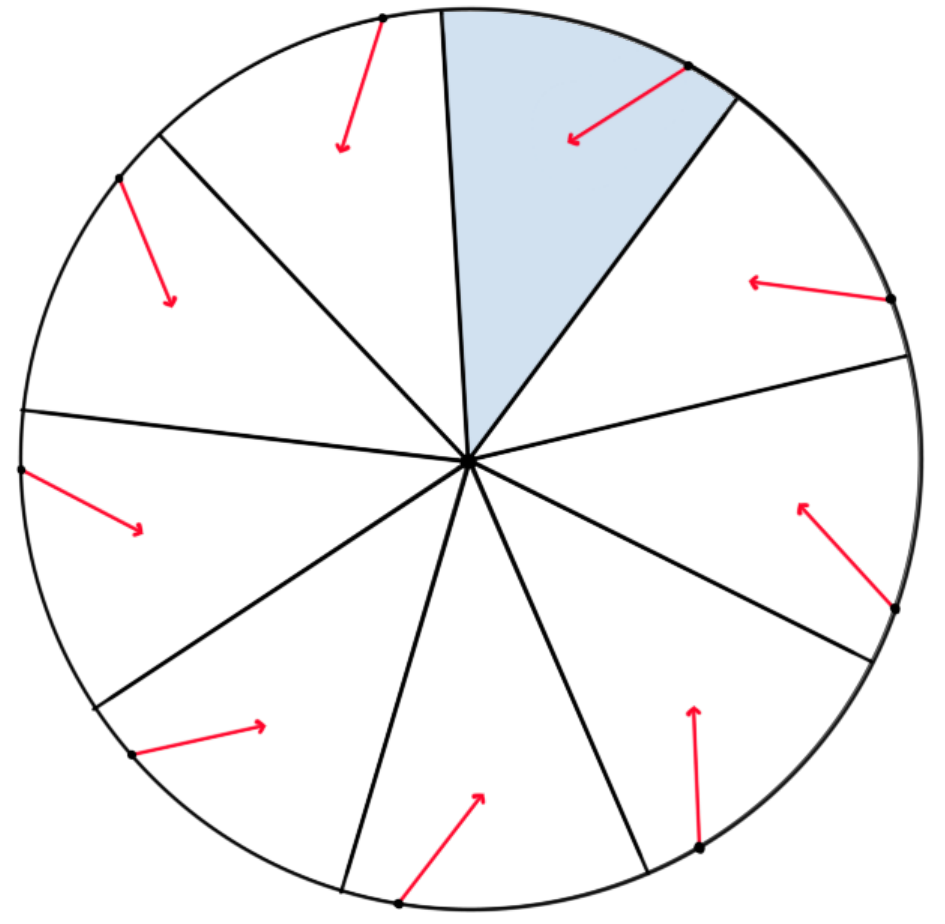}
				\caption*{(b)}
				\vspace*{0.4cm}
			\end{minipage}
			\caption{\textbf{(a)} Section of the $\SP^1$-principal bundle $(\SP^1 \times \SP^1)/H_k \to \SP^1/H_k $. \textbf{(b)} The copy of this section on the  bundle  $\SP^1 \times \SP^1 \to \SP^1$.}
			\label{liftingvectorfield}
		\end{figure}

		Since $s=m_ks_k$ in $H_1\big(\SP(TB),\QQ \big)$,
		$$\sigma|_{\partial B'} =  \Bigg(-f-\sum_{k=1}^n \frac1{m_k}\Bigg)s$$
		and $\chi(B) =f+\sum_{k=1}^n \frac1{m_k} $, where $\sigma|_{\partial D_0} = f s$.
		
		By the same argument, in $H_1(\SP(T\tilde B),\QQ)$ we have
		\[\sigma|_{\partial B'} =  \Bigg(-f-\sum_{k=1}^n 1\Bigg)s\]
		and $\chi(\tilde B) = f+n$. It remains to compare the formulas for $\chi(B)$ and $\chi(\tilde B)$.
	\end{proof}
	
	The following theorem states that Euler numbers of $\SP^1$-orbibundles are not arbitrary rational numbers but form a particular discrete subset of $\RR$.
	
	\begin{thm}\label{thm classification circle bundles} Consider the connected, compact, and oriented $2$-orbifold  $B$ with isolated singularities $x_1,\ldots,x_n$, where $m_k$ is the order of $x_k$. The Euler number of an $\SP^1$-orbibundle $\zeta:M\to B$ belongs to $ \ZZ + \frac1{m_1}\ZZ + \cdots +\frac1{m_n}\ZZ$.

	\end{thm}
	\begin{proof} Let $x_0$ be a regular point in $B$, which has order $m_0=1$. As in the proof of Theorem \ref{eulerchar}, extract from $B$ small open discs $D_k$, $k=0,1,\dots,n$, corresponding to the image of $\BB^2/\Gamma_k$ under the orbifold chart $\phi_k:2\BB^2/\Gamma_k \to 2D_k$ centered at $x_k$, obtaining the surface $B'$ with boundary. The bundle $M|_{B'} \to B'$ is trivial (because $B'$ has the homotopy type of a graph) and, consequently, there exists a global section $\sigma: B' \to M|_{B'}$. If $s_k$ is the fiber over $x_k$, then $\sigma|_{\partial D_k}=q_k s_k$ in homology for some $q_k\in \ZZ$. 
    Considering an arbitrary regular fiber $s$ and the identity $s = \frac{m_k}{d_k} s_k$, with $d_k \in \ZZ$,
    which follows from Lemma \ref{circle section lemma}, we conclude that
		$$\sigma|_{\partial B} = - \sum_{k=0} \frac{d_k q_k}{m_k}s$$
	in $H_1(M,\QQ)$. The result follows from Definition \ref{definition: euler number}.
	\end{proof}
	
	\section{Euler number for orbigoodles}

	An orbifold $B_1$ is a {\bf good orbifold} if it is orbifold covered by a simply connected manifold $\HH$. If $G_1$ is the orbifold fundamental group of $B_1$, i.e., the deck group of a fixed orbifold covering map $\HH \to B_1$, then $B_1 = \HH/G_1$, where $G_1$ acts properly discontinuously on $\HH$.  If $p:B_2 \to B_1$ is an orbifold covering, then there exists $G_2 \subset G_1$ such that $B_2 = \HH/G_2$ and $p$ is the quotient map.
	
	\begin{defi}\label{definition: orbigoodle} Let $V$ be a diffeological space. Consider an action of $G_1$ on $\HH \times V$ by diffeomorphisms such that $g(x,v)=(gx,a(g,x)v)$, where $g\in G_1$ and $(x,v) \in \HH \times V$. An {\bf orbigoodle (good orbibundle)} with fiber $V$ is the natural projection $\zeta_1:L_1 \to B_1$ where $L_1 := (\HH \times V)/G_1$. If $G_2$ is a subgroup of $G_1$, then {\bf the pullback of the orbigoodle}  $\zeta_1$ by the orbifold covering $p:B_2 \to B_1$ is the orbigoodle $\zeta_2:L_2 \to B_1$ given by restricting the action of $G_1$ to $G_2$ with the natural map $P$ according to the diagram
	$$
	\begin{tikzcd}
	L_2 \arrow[r, "P"] \arrow[d, "\zeta_2"'] & L_1 \arrow[d, "\zeta_1"] \\
	B_2 \arrow[r, "p"]                                  & B_1                                
	\end{tikzcd}
	$$
	Sometimes, we denote $L_2$ by $p^\ast L_1$.
	\end{defi}

	In the case $G_1' =h G_1 h^{-1}$, where $h:\HH \to \HH'$ is a diffeomorphism (or, equivalently, $G_1'$ is the deck group of another universal orbifold covering map $\HH' \to B_1$), there is a natural action of $G_1'$ on $\HH' \times V$ which arises from the action of $G_1$ on $\HH \times V$ and is defined by $$g_1'(x,v) := \big(g_1'x,a(h^{-1}g_1'h,h^{-1}g_1'h x)v\big).$$ 
	
	The map $H:(\HH \times V)/G_1 \to (\HH' \times V)/{G_1'}$, given by $H[x,v]:=[hx,v]$, is a bundle isomorphism on the top of $h:\HH /G_1 \to \HH/{G_1'}$. Therefore, an orbigoodle brings forth natural orbigoodles for each universal orbifold covering of $B_1$, and all these orbigoodles are isomorphic.

	\begin{rmk} The orbigoodle structure depends on how $G_1$ acts on $\HH \times V$. If $V$ is an $n$-dimensional vector space and the action is by (orientation preserving) linear isomorphisms, that is, each map $a(g,x):V \to V$ is a linear isomorphism, then we obtain a (oriented) {\bf vector orbigoodle} of rank~$n$. If $V$ is the group $\SP^1$ and the action is by multiplication, then we obtain an {\bf $\SP^1$-orbigoodle}. If $V=\BB^2$ and the action is by diffeomorphisms, then we obtain a {\bf disc orbigoodle}.
	\end{rmk}

	\begin{thm}\label{euler thm} If $B_1$ is a connected, compact, oriented good $2$-orbifold, the orbifold covering map $p:B_2 \to B_1$ has degree $d$, and $\zeta:L \to B_1$ is an oriented vector orbigoodle of rank $2$, then  $$e(p^\ast L)=de(L).$$
	\end{thm}
	\begin{proof} Postponed to the end of the section.
	\end{proof}
	
	A direct consequence of this result is the identity $\chi\big(\HH/G_1\big) = [G_1:G_2] \chi\big(\HH/G_2\big)$.
	
	\begin{defi} If $B$ is a connected, compact, oriented good $2$-orbifold and $L \to B$ is an oriented real vector orbigoodle of rank $2$, then the number $e(L)/\chi(B)$ is called the {\bf relative Euler number} of $L$. It is invariant under pullbacks by orbifold coverings of finite degree.
	\end{defi}
	
	
	Now we examine the behavior of orbigoodles under pullbacks, according to Definition \ref{definition: orbigoodle}. Since $B_1=\HH/G_1$ admits a Riemannian metric, we assume $\HH$ is a Riemannian manifold and $G_1$ is a group of isometries acting properly discontinuously on $\HH$, i.e., the orbits are discrete and the stabilizers of points are finite. Let us locally trivialize the orbigoodle $\zeta_1:L_1 \to B_1$.
	
	Given $x \in \HH$, there exists $r>0$ such that $B(x,r)\cap gB(x,r) = \varnothing$ for every $g\in G_1\setminus \stab_{G_1}(x)$, where $B(x,r)$ stands for the ball in $\HH$ of radius $r$ centered at $x$. We have the trivializations 
	$$
	\begin{tikzcd}
	{(B(x,r)\times V)/\stab_{G_1}(x)} \arrow[d] \arrow[r] & \zeta_1^{-1}(D) \arrow[d, "\zeta_1"] \\
	{B(x,r)/\stab_{G_1}(x)} \arrow[r]            & D                               
	\end{tikzcd}
	$$
	where $D= B(x,r)/\stab_{G_1}(x)$, $\zeta_1^{-1}(D)=(B(x,r)\times V)/\stab_{G_1}(x)$, and the horizontal maps are identities.
	
	Associated to the orbifold covering $p:B_2 \to B_1$, where $B_2 = \HH/G_2$ for a subgroup $G_2$ of $G_1$ and $p$ is the quotient map, we have the pullback $L_2$ of $L_1$ by $p$. The following lemma proves that $p$ is indeed an orbifold cover and lay down the reasoning for why $P: L_2 \to L_1$ behaves as a covering map as well.
	Writing $F:=G_1/G_2$, the disjoint orbits of the action of $G_2$ on $G_1$ on the left, we obtain the model fiber for the covering map $p$ with discrete diffeology. The group $G_1$ acts on $F$ on the right. 
	
	\begin{lemma} \label{lemma: good covering}
	The map $p:\HH/G_2 \to \HH/G_1$ is an orbifold cover with fiber $F$.
	\end{lemma}
	\begin{proof}
	Consider the group $\Gamma: = \stab_{G_1}(x)$ acting on $F$ on the right and the quotient $F/\Gamma = \{f_i \Gamma\}$. Also fix a representative $s_i \in G_1$ of $f_i \in F$ and denote by $\pi_2$ the quotient map $\HH \to \HH/G_2$.
		
	Following the notation used in the previous diagram, the balls $D_i:=\pi_2(B(s_ix,r))$ are pairwise disjoint, because $G_2 s_i = f_i$. Additionally, they do not depend on the representative $s_i$ of $f_i$.
	
	If $g_1 \in G_1$, then $G_2g_1\Gamma$ is one of the $f_i\Gamma$'s. Hence, there are $g_2 \in G_2$ and $h_1 \in \Gamma$ such that $g_2g_1h_1 =s_i$, and $$ \pi_2(B(s_ix,r))= \pi_2(B(g_2g_1h_1x,r)) = \pi_2(B(g_1x,r)).$$
	Therefore,
	$$p^{-1}(D) = \coprod_i\pi_2(B(s_ix,r)).$$	
	Furthermore,
	$B(s_ix,r)/\big(s_i\stab_{\Gamma}(f_i)s_i^{-1}\big) = \pi_2(B(s_ix,r))$
	and we conclude that
	$$
	p^{-1}(D) = \coprod_i B(s_ix,r)/\big(s_i\stab_{\Gamma}(f_i)s_i^{-1}\big).
	$$
	The diffeomorphism $$\psi: \coprod_i B(x,r)/(\stab_{\Gamma}(f_i)) \to \coprod_i B(s_ix,r)/(s_i\stab_{\Gamma}(f_i)s_i^{-1}),$$
	defined by $\psi([z],i) = ([s_iz],i)$ and the natural diffeomorphism
	$$\coprod_i B(x,r)/(\stab_{\Gamma}(f_i)) \simeq (B(x,r) \times F)/\Gamma$$
	produce a local trivialization of $p$
	$$
	\begin{tikzcd}
	(B(x,r) \times F)/\Gamma \arrow[d,"\psi"] \arrow[r] & {B(x,r)/\Gamma} \arrow[d,"\mathrm{id}"] \\
	p^{-1}(D) \arrow[r, "p"']                         & D                                
	\end{tikzcd}
	$$
	proving that $p$ is an orbifold covering map.
\end{proof}
	
	Applying to the map $P$ the same arguments as those in the proof of Lemma \ref{lemma: good covering} we obtain the commutative diagram 
	$$
	\begin{tikzcd}
	{{(B(x,r)\times V \times F)/\Gamma}} \arrow[rd, "\Phi"'] \arrow[ddd] \arrow[rrr] &                                                              &                                      & {(B(x,r)\times V)/\Gamma} \arrow[ld, "\mathrm{id}"] \arrow[ddd] \\
	& \zeta_2^{-1}(p^{-1}(D)) \arrow[d, "\zeta_2"'] \arrow[r, "P"] & \zeta_1^{-1}(D) \arrow[d, "\zeta_1"] &                                                              \\
	& p^{-1}(D) \arrow[r, "p"]                                     & D                                    &                                                              \\
	{(B(x,r)\times F)/\Gamma} \arrow[ru, "\phi"] \arrow[rrr]                         &                                                              &                                      & {B(x,r)/\Gamma} \arrow[lu, "\mathrm{id}"']                     
	\end{tikzcd}
	$$
	
	If $\zeta_1$ is a vector orbigoodle, then $P$ is a legit orbifold cover, and $P$ and $p$ have the same degree~$|F|=[G_1:G_2]$.
	
	\smallskip
	
	\begin{proof}[\bf Proof of Theorem \ref{euler thm}] Let $x_1,\ldots,x_n$ be the singular points of $B_1$ and $x_0$ be a regular point.
	For each $x_k\in B_1$ consider an orbifold chart $\phi_1^k:\BB^2/\Gamma_k \to D_k$ centered at $x_k$ on $B_1$ which trivializes the orbifold covering map $p$ and the real vector orbigoodle $\zeta_1$ of rank $2$ as described in the diagram
	$$
	\begin{tikzcd}
	\coprod_{i=1}^{l_k}(\BB^2\times \RR^2)/\mathrm{stab}_{\Gamma_k}(f_i^k) \arrow[r, "\Phi_2^k"] \arrow[dd, "\mathrm{pr}_1"'] & \zeta_2^{-1}(p^{-1}(D_k)) \arrow[d, "\zeta_2"'] \arrow[r, "P"] & \zeta_1^{-1}(D_k) \arrow[d, "\zeta_1"] & (\mathbb B^2 \times \mathbb R^2)/\Gamma_k \arrow[l, "\Phi_1^k"'] \arrow[dd, "\mathrm{pr}_1"] \\
	& p^{-1}(D_k) \arrow[r, "p"]                                & D_k                                    &                                                                                         \\
	\coprod_{i= 1}^{l_k}\BB^2/\mathrm{stab}_{\Gamma_k}(f_i^k) \arrow[ru, "\phi_2^k"] \arrow[rrr]                               &                                                           &                                        & \mathbb B^2/\Gamma_k \arrow[lu, "\phi_1^k"']                                                
	\end{tikzcd}
	$$
	where $F/\Gamma_k = \{f_i^k \Gamma_k\}_{i=1}^{l_k}$ is the set of disjoint orbits of $F$ given by the action of $\Gamma_k$.

	Consider the surfaces with boundary $B_1' := B_1 \setminus \sqcup_i D_i$ and $B_2' := p^{-1}(B_1')$, and a non-vanishing section $\xi_1:B_1' \to L_1|_{B_1'}$. The fundamental groups identities $$\pi_1(L_1|_{B_1'}) = \pi_1(B_1'),\quad  \pi_1(L_2|_{B_2'}) = \pi_1(B_2')$$ along with 
	$$  (\xi_1\circ p)_\ast (\pi_1(B_2')) \subset P_\ast (\pi_1(L_2|_{B_2'})) \subset \pi_1(L_1|_{B_1'})$$
	guarantees the existence of a section $\xi_2:B_2' \to L_2|_{B_2'}$ on the top of $\xi_1$. Indeed,
	for a point $x \in B_2'$ and a vector $v \in P^{-1}(\xi_1 \circ p(x))$, the map $\xi_1 \circ p:B_2' \to L_1|_{B_1'}$ can be lifted to $\xi_2:B_2' \to L_2|_{B_2'}$ by the covering map $P$ such that $\xi_2(x) = v$:
	
	$$
	\begin{tikzcd}
	& L_2|_{B_2'} \arrow[d, "P"] \\
	B_2' \arrow[r, "\xi_1 \circ p"'] \arrow[ru, "\xi_2"] & L_1|_{B_1'}               
	\end{tikzcd}
	$$
	
	The smooth map $\xi_2:B_2' \to L_2|_{B_2'}$ is a section and satisfy $P\circ \xi_2 = \xi_1 \circ p$.
	Define the sections $\sigma_1 = [\xi_1]$ and $\sigma_2=[\xi_2]$ on the $\SP^1$-bundle $\SP(L_1|_{B_1'})$ and $\SP(L_2|_{B_2'})$.
	
	Let us prove that if $s_1$ and $s_2$ are fibers of $\SP(L_1)$ and $\SP(L_2)$ over regular points, then whenever
	$$
	\sigma_1|_{\partial D_k}= \frac{q_k}{|\Gamma_k|} s_1
	$$
	in $H_1(\SP(L_1),\QQ)$, 
	$$
	\sigma_2|_{p^{-1}(\partial D_k)} =\frac{dq_k}{|\Gamma_k|} s_2
	$$
	in $H_1(\SP(L_2),\QQ)$ and, consequently, we obtain the relation
	$$e(L_2)=de(L_1).$$
	
	Indeed, observe that $$p^{-1}(D_k) = \coprod_{i=1}^{l_k} D(f_i^k),$$ where $D(f_i^k):=\phi_2^k\big(\BB^2/\mathrm{stab}_{\Gamma_k}(f_i^k)\big)$ is a disc with center $\phi_2^k([0])$. 
	
	By the third item on the Lemma \ref{circle section lemma} we have
	
	$$\sigma_2|_{\partial D(f_i^k)} =\frac{q_k}{|\mathrm{stab}_{\Gamma_k}(f_i^k)|} s_2$$
	in $H_1(\SP(L_2),\QQ)$. So,
	$$
	\sigma_2|_{\partial p^{-1}(D_k)} = \sum_{i=1}^{l_k}  \frac{q_k}{|\mathrm{stab}_{\Gamma_k}(f_i^k)|} s_2.
	$$
	Since
	$$
	\sum_{i=1}^{l_k}|\Gamma_k|/|\mathrm{stab}_{\Gamma_k}(f_i^k)| =|F| =d,
	$$
	because $F$ is the disjoint union of the orbits $f_i^k \Gamma_k$,
	we have 
	$$
	\sigma_2|_{\partial p^{-1}(D_k)} =\frac{dq_k}{|\Gamma_k|} s_2.
	$$
\end{proof}

	\section{Euler number via Chern-Weil theory} \label{chernweil}
    A differential $k$-form on a diffeological space $M$ is a function $\omega$ that maps each plot $\phi:U\to M$ to a differential $k$-form $\phi^\ast\omega$ on $U$ satisfying the following property: if $\phi:U \to M$ is a plot and $g:V \to U$ is a smooth function between Euclidean open sets, then $g^\ast(\phi^\ast \omega) = (\phi g)^\ast\omega$. The space $\Omega^k(M)$ of all differential $k$-forms is a real vector space. Moreover, the exterior derivative $d:\Omega^k(M) \to \Omega^{k+1}(M)$ can be defined by $\phi^\ast (\dd\omega) := \dd(\phi^\ast \omega)$. Similarly, the wedge product $\omega_1 \wedge \omega_2$ of two differential forms $\omega_1$ and $\omega_2$ is defined by $\phi^\ast(\omega_1 \wedge \omega_2):= (\phi^\ast\omega_1) \wedge (\phi^\ast \omega_2)$. Following the same reasoning, the pullback $f^\ast \omega$ of a differential form $\omega$ on $N$ under a smooth map $f:M \to N$ is given by $\phi^\ast(f^\ast \omega):=(f\phi)^\ast \omega$.

    Now, consider an $n$-orbifold $B$ and an orbifold chart $\phi:\BB^n/\Gamma \to D$. Denote by $\widetilde \phi:\BB^n \to D$ the map $\widetilde \phi(x) := \phi([x])$. The integral of an $n$-form $\omega$ over $D$ is defined by
	$$\int_D\omega :=\frac{1}{|\Gamma|} \int_{\BB^n} \widetilde{\phi}^\ast \omega.$$
	
	The intuition behind this definition is very simple. For the sake of argument, we assume that the orbifold is two-dimensional and $\Gamma=\langle \exp(2\pi i /m)\rangle$. The action of $\Gamma$ in $\BB^2$ has the sector $S:=\{z\in \BB^2\mid z=0 \text{ or } 0\leq \arg(z)\leq 2\pi/m \}$ as a fundamental domain. Therefore, the quotient $\BB^2/\Gamma$ is the cone obtained by gluing the sides of $S$. Taking a region $R$ in the cone, the inverse image $\widetilde R$ of $R$ by the projection $\BB^2 \to \BB^2/\Gamma$ is made of $|\Gamma|$ copies of $R$. Hence, it is natural to define the area of $R$ as the area of $\widetilde R$ divided by $|\Gamma|$. 
	
	Furthermore, we want to integrate over all the orbifold and, in order to do that, we imitate the definition of integral over smooth manifolds. If $B$ is a compact orbifold, then consider a finite open cover $D_1,\ldots,D_k$ of $B$ such that each $D_i$ comes from an orbifold chart $\phi_i:\BB^n/\Gamma_i \to D_i$. If $\rho_i$ is a smooth partition of unit subordinate to the cover $D_i$, then the integral of an $n$-form $\omega$ over $B$ is defined by
	$$\int_B \omega : =\sum_{i=1}^k \int_{D_i} \rho_i \omega$$
	(see \cite[pag.\,34-35]{alr} and \cite[pag.\,36-37]{car}).

	\begin{thm}\label{integral thm} If $p:B_2 \to B_1$ is an orbifold covering map of degree $d$, where $B_1,B_2$ are compact $n$-orbifolds, and $\omega$ is a differential $n$-form on $B_1$, then
		$$
		\int_{B_2} p^\ast \omega = d \int_{B_1} \omega. 
		$$
	\end{thm}
	\begin{proof}
		The fact follows immediately from the definitions of integral and orbifold cover.
	\end{proof}

	\begin{rmk}\label{important remark 2-orbifold = quotient}
	By \cite[Theorem 13.3.6]{thu}, every connected, compact, oriented $2$-orbifold $B$ with negative Euler characteristic is diffeomorphic to $\HH_\CC^1/G$, where $G$ is a cocompact Fuchsian group and, in particular, $\pi_1^\orb(B) = G$. Moreover, since every cocompact group is finitely generated, $B$ is finitely orbifold covered by a compact surface, because $G$ admits a normal torsion-free finite index subgroup by Selberg's lemma \cite[pag.\,331, Corollary\,5]{ratc}.
	\end{rmk}
	
    Consider a connected, compact, oriented good $2$-orbifold $B$ with negative Euler number. By Remark \ref{important remark 2-orbifold = quotient} above, we can assume $B = \HH_\CC^1/G$ for some Fuchsian group $G$. Consider as well an action of $G$ on $\HH_\CC^1 \times V$ that gives rise, as in Definition \ref{definition: orbigoodle}, to the vector orbigoodle $L:=(\HH_\CC^1 \times V)/G$ over $B$, where $V$ is an $n$-dimensional vector space over $\KK=\RR$ or $\CC$. In the real case we assume the vector orbigoodle to be oriented. 
	
	A connection on the orbigoodle $L \to B$ is a connection $\nabla$ on the trivial bundle $\HH_\CC^1 \times V \to \HH_\CC^1$ which is invariant under the action of $G$, i.e., $g^{-1}\nabla_{gu} gs = \nabla_u s$, where $s$ is a section of the trivial bundle, $u$ is a tangent vector of the hyperbolic plane, and the action of $G$ on sections is given by the formula $gs(x):=g s(g^{-1}x)$.

	\begin{lemma} A Riemannian metric on the vector orbigoodle $L \to B$ admits a metric connection.
	\end{lemma}
	\begin{proof}
    Consider an open cover by orbifold charts $B(x_i,r_i)/\Gamma_i$ of $B$, where $1 \leq i\leq k$ and $\Gamma_i$ is the stabilizer of $x_i$ on $G$. Let $f_i$ be a partition of unity subordinate to this cover. Over the open sets $U_i := \bigsqcup_{x \in Gx_i} B(x,r_i)$, the smooth maps $\rho_i:U_i \to [0,1]$ given by $\rho_i(x) := f_i[x]$ form a partition of unity of $\HH_\CC^1$ subordinate to $U_i$. 
    Since the vector orbibundle has a Riemannian metric, there is a $G$-invariant metric $\langle \cdot,\cdot\rangle$ on the trivial bundle $\HH_\CC^1 \times V \to \HH_\CC^1$.
		
	Fix a metric connection $\widetilde \nabla$ on the trivial bundle $\HH_\CC^1 \times V \to \HH_\CC^1$ and let $\nabla^i$ be a connection on $B(x_i,r_i)\times V \to B(x_i,r_i)$ given by the formula
	$$\nabla_v^i s :=\frac1{|\Gamma_i|} \sum_{h \in \Gamma_i} h^{-1}\widetilde \nabla_{hv}hs.$$
    
    The connection $\nabla^i$ is a metric connection with respect to $\langle \cdot,\cdot  \rangle|_{B(x_i,r_i)}$.
    
	For $B(g^{-1}x_i,r_i)\times V \to B(g^{-1}x_i,r_i)$ define $ \nabla_v^i s :=g^{-1}\nabla_{gv}^igs$. 
	The connection $\nabla^i$ is $G$-invariant on $U_i$ and compatible with the metric $\langle \cdot,\cdot \rangle|_{U_i}$. Hence, the connection 
	$\nabla := \sum_{i=1}^k \rho_i \nabla^i$ is $G$-invariant on the trivial bundle $\HH_\CC^1 \times V \to \HH_\CC^1$ and compatible with the metric $\langle \cdot,\cdot \rangle$.
	\end{proof}

	In what follows, $R$ is the Riemann curvature tensor of a $G$-invariant 
    metric connection 
    $\nabla$.
    
	In the real case with $n$ even, we have the $G$-invariant $n$-form $\pf(R)$, the Pfaffian of the curvature tensor. In the complex case, we have the first Chern number
	$c_1(L) := -\frac{1}{2\pi i} \int_B \tr(R).$

	\begin{thm} \label{chern=euler2} If $\zeta:L \to B$ is an oriented real vector orbigoodle of rank $2$ and $\nabla$ is a metric connection on $L$, then 
		$$e(L) = \frac{1}{2\pi}\int_B \pf(R).$$
	\end{thm}
	\begin{proof} The surface case is well-known (Gauss-Bonnet-Chern theorem). Otherwise, there exists an orbifold covering $p:B' \to B$ of degree $d$, where $B'$ is a compact, oriented surface, and we consider the pullback $p^\ast L$ of $L$ by $p$. From Theorems \ref{euler thm} and \ref{integral thm} we conclude that
		$$e(L) = \frac1d e(p^\ast L) =\frac1d \frac{1}{2\pi}\int_{B'} \pf(R) = \frac{1}{2\pi}\int_B \pf(R).$$
	\end{proof}
	\begin{thm} \label{chern=euler} If $\zeta:L \to B$ is an oriented complex line orbigoodle and $\nabla$ is a Hermitian connection on $L$, then 
		$$e(L) = c_1(L).$$
	\end{thm}
	\begin{proof} Analogous to the proof of Theorem \ref{chern=euler2}.
	\end{proof}

	\newpage
	
	\section{Applications to complex hyperbolic geometry}
	In this section, the $2$-orbifolds are connected, compact, oriented, and have negative Euler characteristic.
	
	\subsection{Complex hyperbolic disc orbigoodles and $\PU(2,1)$-character varieties}
	
	Consider a disc orbigoodle $\zeta:M \to B$ over a $2$-orbifold $B$. The orbifold $M$ is complex hyperbolic if it is diffeomorphic to a $4$-orbifold $\HH_\CC^2/G$, where $G$ is a discrete subgroup of $\PU(2,1)$. Let~$p:\HH_\CC^1 \to B$ be a universal covering map of $B$. The pullback of $\zeta$ by $p$ provides the universal cover of $M$ as~$\HH_\CC^1 \times \BB^2$ with deck group $\pi_1^\orb(B)$. Therefore, the orbifold fundamental group of $M$ is isomorphic to $\pi_1^\orb(B)$ and, consequently, we obtain a discrete faithful representation $\rho: \pi_1^\orb(B) \to \PU(2,1)$ such that $\HH_\CC^2/\pi_1^\orb(B) = M$. Hence, complex hyperbolic disc orbigoodles arise from particular discrete faithful representations of $\pi_1^\orb(B)$ in $\PU(2,1)$. 
	
	The orbifold $\HH_\CC^2/G$ has a natural geometric structure. The Riemannian metric $g$ of $\HH_\CC^2$ can be induced in the $4$-orbifold because $g$ is $\PU(2,1)$-invariant. Furthermore, we say that $\HH_\CC^2/G_1$ and $\HH_\CC^2/G_2$ are isometric if there is an orientation preserving diffeomorphism $f:\HH_\CC^2/G_1 \to \HH_\CC^2/G_2$ such that $f^\ast g = g$. The quotient maps $P_i:\HH_\CC^2 \to \HH_\CC^2/G_i$ (which are universal covers) provide two universal covering maps $P_2$ and $P_2':=f \circ P_1$ for $\HH_2/G_2$. Hence, there exists a diffeomorphism $F:\HH_\CC^2 \to \HH_\CC^2$ such that the diagram
	$$
	\begin{tikzcd}
	\HH_\CC^2 \arrow[d, "P_1"] \arrow[rr, "F"] &  & \HH_\CC^2 \arrow[d, "P_2"] \\
	\HH_\CC^2/G_1 \arrow[rr, "f"]             &  & \HH_\CC^2/G_2            
	\end{tikzcd}
	$$
	commutes and the map $F$ is clearly an isometry. If we have a diffeomorphism $\HH_\CC^2/G_1 \simeq M$ then, by composing with the isometry $f$, we obtain $\HH_\CC^2/G_2 \simeq M$. The corresponding discrete faithful representations $\rho_i:\pi_1^\orb(B) \to \PU(2,1)$, $i=1,2$, have images $G_i$ and satisfy $\rho_2(\cdot) = F\rho_1(\cdot)F^{-1}$. Therefore, isometric complex hyperbolic structures on $M$ correspond to the same representation up to conjugation in $\PU(2,1)$.
	
	\begin{defi}\label{character variety} The {\bf $\PU(2,1)$-character variety} of the connected, compact, oriented $2$-orbifold $B$ with negative Euler characteristic is the space
		$$
		\hom\big(\pi_1^\orb(B),\PU(2,1)\big)/\PU(2,1),
		$$
		where $\hom\big(\pi_1^\orb(B),\PU(2,1)\big)$ is the space of all group homomorphisms $\pi_1^\orb(B) \to \PU(2,1)$ on which $\PU(2,1)$ acts by conjugation.
	\end{defi}
	
	It is clear from the above discussion that complex hyperbolic disc orbigoodles over $B$, considered up to isometry, can be seen as points in the $\PU(2,1)$-character variety of $B$. Nevertheless, not all representations correspond to complex hyperbolic orbigoodles (say, there are faithful non-discrete representations).
	
	\subsection{Discreteness of the Toledo invariant, holomorphic section identity and Toledo rigidity}

	\begin{lemma}\label{lemaequivariante} Consider a $2$-orbifold $\HH_\CC^1/G$ and a representation $\rho:G \to \PU(2,1)$, where $G$ is a Fuchsian group. There exists a smooth $G$-equivariant map $\HH_\CC^1 \to \HH_\CC^2$. Furthermore, if we have two such $G$-equivariant maps $f_0,f_1$, then there exists a smooth homotopy $f_t$, $t\in \RR$, such that $f_t$ is $G$-equivariant for every $t$.
	\end{lemma}
	\begin{proof}
	Let $[x_1],\ldots,[x_n]$ be the singular points of $\HH_\CC^1/G$, consider the orbigoodle $$\zeta:(\HH_\CC^1 \times \HH_\CC^2)/G \to \HH_\CC^1/G$$ and observe that each point $x_i$ of $\HH_\CC^1$ satisfies $\mathrm{stab}_{G}(x_i) \neq {1}$. Take a small geodesic disc $B(x_i,r_i)$ centered at $x_i$ and let $D_i$  be its image under the quotient map $\HH_\CC^1 \to \HH_\CC^1/G$. The trivialization of $\zeta$ is given by the commutative diagram
	$$
	\begin{tikzcd}
	\big(B(x_i,r_i)\times \HH_\CC^2\big)/\mathrm{stab}_{H}(x_i)  \arrow[r, "\Phi_i"] \arrow[d]       & \zeta^{-1}(D_i) \arrow[d, "\zeta"] \\
	B(x_i,r_i)/\mathrm{stab}_{H}(x_i) \arrow[r, "\phi_i"] & D_i                                 
	\end{tikzcd}
	$$
		
	Since $G$ is Fuchsian, the subgroup $\mathrm{stab}_G(x_i)$ is finite and, therefore, it is generated by a rotation $R_i$. Then $I_i:=\rho(R_i)$ has finite order and, consequently, is an elliptic isometry or the identity. Let $p_i$ be a fixed point of $I_i$ and define the map $$F_{x_i}: B(x_i,r_i)/\mathrm{stab}_{G}(x_i) \to(B(x_i,r_i)\times \HH_\CC^2)/\mathrm{stab}_{G}(x_i)$$ by the formula $F_{x_i}([x]) = ([x,p_i])$, producing a section on the neighborhood of each $[x_i]\in \HH_\CC^1/G$ as a consequence of the trivialization described in the above diagram. Now, we extend these sections constructed around each singular point to the entire orbifold  (a standard argument for fiber bundles over manifolds). Hence, we have a global section $F:\HH_\CC^1/G \to (\HH_\CC^1\times \HH_\CC^2)/G$. Considering $\widetilde F:\HH_\CC^1 \to (\HH_\CC^1\times \HH_\CC^2)/G$ given by $\widetilde F(x) = F([x])$ and writing $\widetilde F(x)=[x,f(x)]$ we obtain a $G$-equivariant smooth function $f:\HH_\CC^1 \to \HH_\CC^2$.
		
	Now, we consider two $G$-equivariant maps $f_0,f_1$ which define sections $$F_k:\HH_\CC^1/G \to (\HH_\CC^1\times \HH_\CC^2)/G$$ by the formula $F_k([x]) = ([x,f_k(x)])$. From the trivialization described above, $F_k$ is given, around the singular point $[x_i]$, by the map $\widetilde F_{k,i} = \Phi_i^{-1}\circ F_k \circ \phi_i$. Note that $f_0(x_i)$ and $f_1(x_i)$ are fixed points of $I_i$. Furthermore, deforming the sections we can assume that $\tilde F_{k,i}([x]) =[x,f_k(x_i)]$. 
		
	The set of fixed points of $I_i$ is always connected (see Subsection \ref{subsection complex hyperbolic geometry}): it can be a point or a complex geodesic when $I_i$ is elliptic, or the whole space when $I_i = 1$. Hence we can, in a small neighborhood of $[x_i]$, deform $F_0$ to $F_1$ smoothly and therefore assume $F_1=F_0$ near the singular points. The rest of the deformation is made in a smooth manifold, which is possible since $\HH_\CC^2$ is a ball. So, we have a homotopy $F_t$ between $F_0$ and $F_1$ and, thus, a homotopy $f_t$ between $f_0$ and $f_1$ such that $f_t$ is $G$-equivariant for all $t\in \RR$.
	\end{proof}
	
	Note that $\HH_\CC^2/\pi_1^\orb(B)$ is a diffeological space (but not necessarily  an orbifold!) and the Kähler form $\omega$ is well defined on it because it is invariant under the action of $\pi_1^\orb(B)$.
	
	By Lemma \ref{lemaequivariante}, there exist smooth maps $f:B \to \HH_\CC^2/\pi_1^\orb(B)$ satisfying the following property: for each universal orbifold covering map $\HH_\CC^1 \to B$ with deck group $G$, the map $f$ can be lifted to a smooth $G$-equivariant map $\tilde f: \HH_\CC^1 \to \HH_\CC^2$ such that the diagram
	$$
	\begin{tikzcd}
	\HH_\CC^1 \arrow[d] \arrow[rr, "\tilde f"] &  & \HH_\CC^2 \arrow[d] \\
	\HH_\CC^1/G \arrow[rr, "f"']               &  & \HH_\CC^2/G        
	\end{tikzcd}
	$$
	commutes. We call these functions {\bf good smooth maps} from $B$ to $\HH_\CC^2/\pi_1^\orb(B)$.
	
	\medskip
	\noindent{\bf Question:} Are all smooth maps $B \to \HH_\CC^2/\pi_1^\orb(B)$ good? The answer is likely negative. For the group $\Gamma= \langle \exp(2 \pi i /m) \rangle$, $m\geq 2$, the example $25$ of the paper \cite{IKZ} provides a smooth map $\CC/\Gamma \to \CC/\Gamma$ which does not admit $\Gamma$-equivariant lifts for any group morphism $\Gamma\to \Gamma$.

	\begin{defi}\label{toledo invariant} Let $B$ be a connected, compact, oriented $2$-orbifold with negative Euler characteristic. The {\bf Toledo invariant} of a representation $\rho:\pi_1^\orb(B) \to \PU(2,1)$ is given by the integral
		$$\tau(\rho) := \frac{4}{2\pi} \int_B f^\ast\omega,
		$$
		where $f:B \to \HH_\CC^2/\pi_1^\orb(B)$ is a good smooth map and $\omega$ is the Kähler form on $\HH_\CC^2/\pi_1^\orb(B)$.
	\end{defi}
    
    Regarding the definition of the Toledo invariant, see also \cite{krebs}.

	\begin{lemma} The definition of the Toledo invariant does not depend on the choice of $f$.
	\end{lemma}
	\begin{proof} 
		We take $B= \HH_\CC^1/G$, where $G$ is a Fuchsian group.

		Consider two $G$-equivariant maps $f_0,f_1:\HH_\CC^1 \to \HH_\CC^2$. By Lemma \ref{lemaequivariante} there exists a homotopy $f_t$ between these maps such that $f_t$ is $G$-equivariant for all $t$. 
		
		For each singular point $x_i\in B$ consider the chart  $\phi_i:2\BB^2/H_i \to 2D_i$ centered at $x_i$, where~$2D_i$ is sufficiently small. Removing the discs  $D_i:=\phi_i\big(\BB^2/H_i\big)$ we obtain the open surface $B'$. Applying Stokes theorem
		\begin{equation}\label{toledobemdefinido1}
		0=\int_{B' \times [0,1]} \dd(f_t^\ast \omega ) = \int_{B'}f_1^\ast \omega -    \int_{B'}f_0^\ast \omega + \int_{\partial B' \times [0,1]}f_t^\ast \omega
		\end{equation}
		because $\dd(f_t^\ast \omega ) = f_t^\ast \dd\omega  = 0$ (the Kähler form $\omega$ is closed).

		For each chart $\phi_i$ we have
		$$
		\int_{\overline D_i}f_t^\ast \omega =\frac{1}{|H_i|}\int_{\overline{\BB}^2} (f_t \circ \widetilde{\phi_i})^\ast \omega,  
		$$
		where $\widetilde \phi_i(x): = \phi_i([x])$. By Stokes theorem,
		\begin{equation}\label{toledobemdefinido2}
		0=\frac{1}{|H_i|}\int_{\overline \BB^2 \times [0,1]}\dd(f_t \circ \widetilde{\phi_i})^\ast \omega   = \int_{D_i}f_1^\ast \omega -\int_{D_i}f_0^\ast \omega +  \int_{\partial D_i \times [0,1]}f_t^\ast \omega.
		\end{equation}
		
		By equations \eqref{toledobemdefinido1} and \eqref{toledobemdefinido2} we conclude
		$$
		\int_B f_0^\ast \omega = \int_B f_1^\ast \omega.
		$$  
	\end{proof}
	
	\begin{defi} The relative Toledo invariant  of a representation $\rho:\pi_1^\orb(B)\to \PU(2,1)$, denoted by $\tau_R(\rho)$, is the number $\tau(\rho)/\chi(B)$. This number is unchanged under a finite cover $\widetilde B$ of $B$ by Theorems \ref{integral thm} and $\ref{euler thm}$, i.e.,
		$$
		\frac{\tau(\rho)}{\chi(B)} = \frac{\tau\left(\rho|_{\pi_1^\orb\big(\widetilde B\big)}\right)}{\chi\big(\widetilde B\big)}.
		$$ 
	\end{defi}

	The following Theorem \ref{cherntoledo} and Corollary \ref{discretetau} are generalizations for orbifolds of results proved in~\cite{GKL}. They establish the integrality property of the Toledo invariant.
	
	Let $B= \HH_\CC^1/G$, where $G$ is a Fuchsian group. Consider a representation $\rho:G \to \PU(2,1)$ and a smooth $G$-equivariant map $f:\HH_\CC^1 \to \HH_\CC^2$. Let $f^\ast \TT\HH_\CC^2 \to \BB^2$ be the pullback of the tangent bundle $\TT\HH_\CC^2$ by $f$. We have the vector orbigoodle $\zeta:E \to B$, where $E:= (f^\ast \TT\HH_\CC^2)/G$.

	\begin{thm}\label{cherntoledo} The following formula holds
		$$c_1(E) =\frac 32 \tau(\rho).$$
	\end{thm}
	\begin{proof} A little modification of the argument in the proof of Lemma \ref{lemaequivariante} allows us to assume that $f$ is an immersion out of the singular points. We can pullback the metric, connection, and Kähler form of $\TT\HH_\CC^2$ to $f^\ast \TT\HH_\CC^2$, and since these objects are invariant under the action of $\PU(2,1)$, we conclude that they are well defined in $E$. The first Chern number of $E$ is given by 
	$$c_1(E):=-\frac1{2\pi i}\int_B \eta$$ 
	where $\eta = \tr (R)$ and $R$ is curvature tensor of $(f^\ast \TT\HH_\CC^2,\nabla)$.
		
		In order to calculate the curvature tensor at a regular point we can assume that we are dealing with an embedded surface $S\subset \HH_\CC^2$ and are calculating the curvature tensor of $\TT\HH_\CC^2|_{S}$,  because $f$ is a local embedding around regular points.
		So, take $\pp \in S$. Let $t_1,t_2\in \TT_{\pp} S$ be unit tangent vectors to $S$ at $\pp$ orthogonal with respect to the Riemannian metric $g_{\pp}$. Consider a unit tangent vector $n\in \TT_{\pp}\HH_\CC^2$ orthogonal to $t_1$ with respect to the Hermitian form $h_{\pp}$. Since $t:=t_1$ and $n$ span $\TT_{\pp}\HH_\CC^2$ as a complex vector space and $\real h_{\pp}(t_1,t_2) = g_{\pp}(t_1,t_2) =0$, we have 
		$t_2 = ia t + bn$, with $a\in \RR$  and $b \in \CC$. The curvature tensor $R$ on $\TT_{\pp} \HH_\CC^2$ is given by the formula \eqref{riemanntensor}:
		
		\begin{align*}R(t_1,t_2)t &= 4a it+ bn,
		\\R(t_1,t_2)n &= -\overline b t + 2a i n.
		\end{align*}	  
		Hence, $\tr(R)(t_1,t_2) = 6a i$.
		On the other hand, $\omega_{\pp}(t_1,t_2) =   \imag h_{\pp}(t_1,t_2) = -a$
		and we conclude that $\tr(R) =- 6i \omega$ on the surface $S$. Therefore,
		$c_1(E) = \frac6{2\pi }\int_B f^\ast\omega$ and the result follows.
	\end{proof}

	\begin{cor}\label{discretetau} If the orbifold $B$ has singularities $x_1,\ldots, x_n$ of order $m_1, \ldots, m_n$, then
		$$\tau(\rho) \in\frac 23 \Big(\ZZ + \frac 1{m_1}\ZZ+ \cdots +\frac 1{m_n}\ZZ\Big).$$
		In particular, the Toledo invariant is a discrete invariant.
	\end{cor}
	\begin{proof} By Theorem \ref{cherntoledo} we have $\tau(\rho) = \frac{2}{3}c_1(E)$. Observe that $E$ is a complex vector orbigoodle of rank $2$ and, thus, $c_1(\wedge^2 E) = c_1(E)$. Since $\wedge^2 E$ is a complex line orbibundle, the number $c_1(E)$ belongs to  $\ZZ + \frac 1{m_1}\ZZ+ \cdots +\frac 1{m_n}\ZZ$ by Theorem \ref{chern=euler}.
	\end{proof}
	
	\begin{rmk}
	In the above proof, we used $\wedge^2 E$. One way of building this line bundle is by pulling back $E$ to the universal cover of $B$, taking the second wedge power and then quotient it back to~$B$.
	
	It is worth noting that the above result holds for $\HH_\CC^k$ as well, where the Toledo invariant belongs to $$\frac2{k+1}\Big(\ZZ + \frac 1{m_1}\ZZ+ \cdots +\frac 1{m_n}\ZZ\Big),$$ and the proof is analogous.
	\end{rmk}

	\begin{cor}\label{holomorphicsectionidentity} If a complex hyperbolic disc orbigoodle $M\to B$ has a holomorphic section, i.e., a section $B \to M$ originated from a $\pi_1^\orb(B)$-equivariant holomorphic embedding $\HH_\CC^1 \to \HH_\CC^2$, then $$\frac32\tau_R(M) = e_R(M) +1.$$ 
	\end{cor}
	\begin{proof} We can suppose that $B$ is a surface and it is embedded in $M$ as a Riemann surface because the identity is between relative invariants. Take a point $x_0 \in B$, remove a small disc $D_0$ centered at $x_0$, and consider $B'=B \setminus D_0$. Let $t$ and $s$ be sections of $\SP(\TT B')$ and $\SP(L|_{B'})$, where $L$ is the kernel of the vector bundle morphism $TM|_B \to TB$. Note that $L \to B$ is the vector bundle whose fibers are the tangent planes of the fibers of the disc bundle $M \to B$ at the points where they intersect $B$. Since $M$ is a Riemannian manifold, we suppose these circle bundles are unit bundles. As $B$ is a Riemann surface, $t \wedge s$ is non-zero in $\wedge^2 \TT M|_{B'}$, because $t$ and $s$ are never $\CC$-linear dependent. Take local sections $t_0$ and $s_0$ of $\TT B$ and $L$ in a neighborhood of $D_0$ (that we shrink if necessary) such that $t_0\wedge s_0$ is non-zero at every point, i.e., the vectors $t_0$ and $s_0$ are never in the same complex line. Writing $t = f t_0$ and $s=gs_0$ on $\partial D_0$, where $f,g$ are maps from $\partial D_0$ to $\SP^1$, we obtain $t \wedge s = fg t_0 \wedge s_0$ and, therefore, $e(\wedge^2 \TT M|_B)=e(\TT B)+e(L)$ because $\deg(fg) = \deg(f)+\deg(g)$. By Theorem \ref{cherntoledo}, we have $\frac 32 \tau_R(\rho) =e_R(M)+1$. 
		
	\end{proof}
	
	The following is the generalization for $2$-orbifolds of the classical Toledo rigidity theorem \cite{tol}.
	\begin{thm}\label{toledo rigidity} {\bf (Toledo Rigidity)} The inequality
		$$|\tau_R(\rho)| \leq 1$$
	always holds for representations $\rho:\pi_1^\orb(B) \to \PU(2,1)$. Furthermore, the relative Toledo number is $1$ in absolute value if, and only if, there is a stable complex geodesic.
	\end{thm}

	\begin{proof} Since the inequality holds for surfaces and the relative Toledo number is unchanged under finite orbifold covers, the inequality is true for orbifolds as well. 
		
	The only thing we have to prove is that if $|\tau_R(\rho)|=1$, then there exists a stable complex geodesic. We can assume $B= \HH_\CC^1/G$, where $G$ is a Fuchsian group.
	
	By Remark \ref{important remark 2-orbifold = quotient} there is a normal torsion-free finite index subgroup $H$ of $G$ and by Toledo's rigidity theorem for surfaces there exists a Riemann-Poincaré sphere $L$ stable under the action of~$H$. We write $G/H = \{Hg_1,\ldots,Hg_n\}$, where $H$ acts on $G$ on the left. The projective lines $L$ and $g_iL$ intersect at a point $p_i$. The line $L$ is broken into two Poincaré discs and, therefore, $p_i$ is in one of the Poincaré discs or in their common boundary. The action of $H$ on this Poincaré disc is faithful and discrete by Goldman's theorem (see \cite[Theorem A]{goldmanthesis}).
	
	For each $h \in H$ we have $g_ihg_i^{-1}p_i \in g_iL$, because $p_i \in g_iL$. On the other hand, $g_i h g_i^{-1} p_i \in L$ because $H$ is normal. So, assuming $L\ne g_iL$, we obtain that the group $H$ is in the stabilizer of $p_i$ which is impossible because $H$ is the fundamental group of a surface of genus $\geq2$. So, $L=g_iL$ and the complex geodesic $L$ is stable under $G$.
	\end{proof}

    E. Xia has shown in \cite[Theorem 1.1]{xia} that, given a compact, connected, and oriented surface $B$ with negative Euler characteristic, the number of connected components of the $\PU(2,1)$-character variety of $B$ is the number of $\tau$'s satisfying $$\tau \in\frac23 \ZZ \quad  \text{ and }\quad \Bigg|\frac{\tau}{\chi(B)}\Bigg| \leq 1.$$
	
	Analogously, it is interesting to ask if the same holds in the case of orbifolds. More precisely, given a compact, connected, and oriented $2$-orbifold $B$ with negative Euler characteristic, we conjecture that the number of connected components of the $\PU(2,1)$-character variety of $B$ equals the number of $\tau$'s satisfying $$\tau \in \frac23\bigg(\ZZ+ \frac1{m_1}\ZZ+\cdots+\frac1{m_n}\ZZ\bigg) \quad  \text{ and }\quad \Bigg|\frac{\tau}{\chi(B)}\Bigg| \leq 1.$$
	This conjecture is supported by the following sketchy argument. The map associating each component of the $\PU(2,1)$-character variety to $\tau(\rho)$, where $\rho$ is a representation in that component, is well-defined because whenever we deform $\rho$ along a curve, we are at some level continuously moving the fixed points of some elliptic isometries. For each instant of this deformation, we can construct an equivariant map (with respect to the current representation) accordingly to Lemma~\ref{lemaequivariante} and we may assume that, as we deform the representation, we are simultaneously deforming the corresponding equivariant map. Hence, the Toledo invariant varies continuously along such curve; nevertheless, the Toledo invariant is discrete by Corollary \ref{discretetau} and so it must be constant along the curve.

\end{document}